\documentclass[10pt]{amsart}
\usepackage{amsmath, amsthm, amscd, amssymb}
\usepackage{pinlabel}
\usepackage{hyperref}
\usepackage{subfig}

\newtheorem{theorem}{Theorem}[section]
\newtheorem{lemma}[theorem]{Lemma}
\newtheorem{proposition}[theorem]{Proposition}
\newtheorem{corollary}[theorem]{Corollary}
\newtheorem{conjecture}[theorem]{Conjecture}

\newtheorem*{theorem*}{Theorem}
\newtheorem*{proposition*}{Proposition}
\newtheorem*{main_scylla_thm}{Theorem~\ref{thm:main_scylla}}

\theoremstyle{definition}
\newtheorem{definition}[theorem]{Definition}

\theoremstyle{remark}
\newtheorem{remark}[theorem]{Remark}

\def\Z{\mathbb Z}
\def\R{\mathbb R}
\def\Q{\mathbb Q}

\def\N{\mathbb N}

\def\PP{\mathcal{P}}
\def\EE{\mathcal{E}}
\def\GG{\mathcal{G}}

\def\cl{\textnormal{cl}}
\def\scl{\ensuremath{\textnormal{scl}}}

\def\PSL{\textnormal{PSL}}

\def\hs1t{\ensuremath{\textnormal{Homeo}^+(S^1)^\sim}}

\makeindex

\begin{document}

\title{Stable commutator length in free products of cyclic groups}

\author{Alden Walker}
\address{Department of Mathematics \\ University of Chicago \\
Chicago, IL  60637}
\email{akwalker@math.uchicago.edu}

\begin{abstract}
We give an algorithm to compute stable commutator length 
in free products of cyclic groups which is polynomial time 
in the length of the input, the number of factors, 
and the orders of the finite factors.  We also 
describe some experimental and theoretical applications of this algorithm.
\end{abstract}

\maketitle


\section{Introduction}

\subsection{Background}

Low-dimensional topology is informed by experiment, 
and it is often the case that, besides being 
useful to compute examples, explicit algorithms provide 
important theoretical tools.  In this paper, 
we describe an algorithm to compute stable commutator 
length ($\scl$) in free products of cyclic groups.  

Briefly, if $G$ is a group and $w \in [G,G]$, then 
the \emph{commutator length} $\cl(w)$ is the least 
number of commutators whose product is $w$, and 
the \emph{stable commutator length} 
$\scl(w) = \lim_{n\to\infty}\cl(w^n)/n$.  Stable commutator 
length extends to a function on $B_1(G)$, the space of real 
$1$-boundaries in the group homology of $G$.  
There is another, equivalent, definition of stable commutator 
length in terms of topologically minimal surface maps, which 
will be our main interest.

Even for simple spaces, $\scl$ can be quite complicated, as we describe 
below.  An algorithm to compute $\scl$ in free groups, 
\texttt{scallop}, by Danny Calegari is described in \cite{Calegari_scl}. 
The key to \texttt{scallop} is that surface maps into free groups can be 
parametrized by the combinatorial data of a \emph{fatgraph}.
The papers 
\cite{CW_random_groups, CW_rigidity, CW_surface_subgroups, CW_isometry, CWilton_graphs} 
give theoretical results which leverage the same combinatorial ideas.

\subsection{Result}

The complexity of the \texttt{scallop} algorithm is polynomial in the input 
length and exponential in the rank of the free group (see~\cite{Calegari_scl}).  
In this paper, we describe a generalization, \texttt{scylla}, of the \texttt{scallop} 
algorithm.  The new \texttt{scylla} algorithm computes $\scl$ in 
free products of cyclic groups, both finite and infinite, and is 
polynomial time in the length of the input, the rank of the free factor, the sizes of the 
finite factors, and the total number of factors.  We should remark that 
free products of cyclic groups are virtually free, and it is 
theoretically possible to lift the $\scl$ computation to a finite index 
subgroup and use \texttt{scallop}.  However, this is infeasible: 
if $K$ is the product of the orders of the finite factors, then 
the rank of the finite index free subgroup is, in general, at least $K$.  
It is also necessary to multiply the length of the input by $K$.  
The resulting algorithm therefore 
has complexity at least $O((KI)^{K})$, where $I$ is the 
length of the original input.  It is very infeasible in practice.
By working in the group itself instead of lifting to a finite 
index subgroup, we can achieve polynomial complexity, and 
the algorithm more naturally reflects the group structure, which is 
useful theoretically.

Let $G = *_j G_j$ be a free product of the cyclic groups $G_j$.  Let $o_j$ 
be the order of the generator of $G_j$, where by convention $o_j=0$ if 
$G_j$ is infinite. Given $\Gamma =\sum_i w_i\in B_1(G)$ a collection of reduced, 
minimal length words $w_i$, we define $|\Gamma|$ to be the sum of the 
word lengths of the $w_i$.  We prove the following.

\begin{main_scylla_thm}
In the above notation, the stable commutator length $\scl\left(\sum_ig_i\right)$ is the 
solution to a rational linear programming problem 
with at most $|\Gamma|^3(1+\sum_j o_j) + |\Gamma|^2$ columns and $|\Gamma|^2(1+\sum_jo_j)$ rows.  
An extremal surface map for $\Gamma$ can be constructed from a minimizing vector.
\end{main_scylla_thm}

The linear programming problem dimensions given in the theorem are on
the correct order, but we remark that they are almost certainly overestimates.
Precise dimensions can be given in closed form in terms of how many letters in the 
$w_i$ are in each free factor, but this is not particularly illuminating.

The main contribution of this paper is the idea that surfaces 
can be built out of different kinds of pieces, and as far as computational 
complexity is concerned, the smaller the better.

\subsection{Overview of the paper}

In Section~\ref{sec:top_min}, we review topologically minimal surfaces 
and the definition of $\scl$.  In Section~\ref{sec:free_groups}, we 
show how surface maps into free groups are carried by the 
combinatorial structure of a fatgraph and how this 
structure can be used to compute $\scl$.  Section~\ref{sec:free_prods} 
extends this to give a combinatorial parametrization of
surface maps into free products of cyclic groups.  
Finally, Section~\ref{sec:extensions} shows how one can use the 
\texttt{scylla} algorithm, both experimentally and theoretically.

\subsection{Software}

The algorithm described in this paper is implemented as part 
of the \texttt{scallop} package~\cite{scallop}.  The \texttt{scallop} 
package contains the algorithms \texttt{scallop} and \texttt{scylla}, 
among other things.

\subsection{Acknowledgements}

I would like to thank Danny Calegari and Mark Sapir.  Alden Walker was supported 
by NSF grant DMS 1203888. 

\section{Topologically minimal surfaces}
\label{sec:top_min}

\subsection{Definition}

Let $X$ be a toplogical space and let $\Gamma:\coprod_i S_i^1 \to X$ 
be a collection of loops in $X$ which together are homologically trivial.
If $S$ is a surface with boundary, and $f:S \to X$ is a continuous map, 
we say that the pair $(S,f)$ is an \emph{admissible map} for $\Gamma$ if 
the diagram
\[
\begin{CD}
  S    @>f>> X \\
 @AAA      @AA\Gamma A \\
 \partial S @>>\partial f> \coprod_iS_i^1 \\
\end{CD}
\]
commutes and if, on homology, 
$\partial f_*([\partial S]) = n(S,f)\left[\coprod_iS_i^1\right]$.
That is, if $f$ takes $\partial S$ to the collection of loops 
in $X$ (so it factors through $\Gamma$), and $\partial f$ maps 
to each component of $\Gamma$ with the same degree, which we denote 
by $n(S,f)$.  There is no requirement that $S$ be connected.  

Since $S$ has boundary, genus is not a good measure of complexity, 
so for a connected surface, we define $\chi^-(S) = \min(0, \chi(S))$.  
For a general surface, $\chi^-(S)$ is the sum of $\chi^-$ over the 
connected components of $S$.  Then we define
\[
\scl(\Gamma) = \inf_{(S,f)}\frac{-\chi^-(S)}{2n(S,f)},
\]
where the infimum is taken over all surface maps $(S,f)$ admissible for 
$\Gamma$.  Intuitively, $\scl$ measures the complexity of the most 
``efficient'' surface which bounds $\Gamma$, where the surfaces are 
allowed to map with high degree if that can reduce the average 
Euler characteristic.

Now let $G$ be a group.  Let $B_1(G)$ be boundary chains in the 
group homology $H_1(G; \Z)$, and define 
\[
B_1^H(G) =  B_1(G) \otimes \R / \langle g^n=ng, hgh^{-1}=g\rangle.
\]
That is, $B_1^H(G)$ is the vector space of homologically trivial 
$1$-chains in $G$, where we have taken the quotient to make 
conjugation trivial and make taking powers the same as multiple 
copies.
Now say $\sum_iw_i \in B_1^H(G)$;  
for example, we might take a single element $w \in [G,G]$.
Then $\scl\left(\sum_iw_i\right) = \scl(\Gamma)$, where 
$\Gamma$ is a collection of loops in a $K(G,1)$ space representing the $w_i$.
Note that while the definition of $\scl$ uses specific representative loops 
in $B_1(G)$, it naturally descends to the quotient $B_1^H(G)$, because 
it obviously depends only on the free homotopy classes of the boundaries, 
and multiple boundary components of a surface $S$ mapping to 
the same loop can be joined together with $1$-handles to create a single boundary 
mapping to a power.  It is a proposition that this definition of $\scl$ 
and the group-theoretic definition in the introduction are equivalent.
See \cite{Calegari_scl}, Chapter~2 for a more 
thorough introduction to $\scl$.

Since we allow admissible surfaces to map to $\Gamma$ with 
arbitrary degree, $\scl$ need not be rational, and in fact there 
exist finitely presented groups containing elements 
with transcendental $\scl$ \cite{Zhuang_irrational}.  Even for rational 
$\scl$ values, though, there need not exist a particular 
surface map which realizes the infimum.  
If there \emph{is} such a surface, we say that it is \emph{extremal} 
for $\Gamma$.

\subsection{Experiments in free groups}

The $\scl$ spectrum is quite rich, even (especially?) for free groups.
Using \texttt{scallop} or \texttt{scylla}, it is possible to 
compute the $\scl$ of many random words.  A histogram of the values 
with large bins looks approximately Gaussian, while a histogram 
using small bins shows the fractal-like nature of the spectrum.  
See Figure~\ref{fig:histogram}.

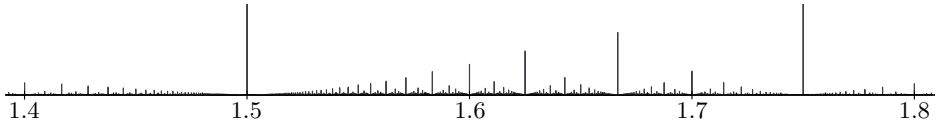
\begin{figure}[ht]
\begin{center}
\labellist
\small\hair2pt
\pinlabel $1.4$ at 13 -8
\pinlabel $1.5$ at 171 -8
\pinlabel $1.6$ at 330 -8
\pinlabel $1.7$ at 488 -8
\pinlabel $1.8$ at 646 -8
\endlabellist
\includegraphics[scale=0.54]{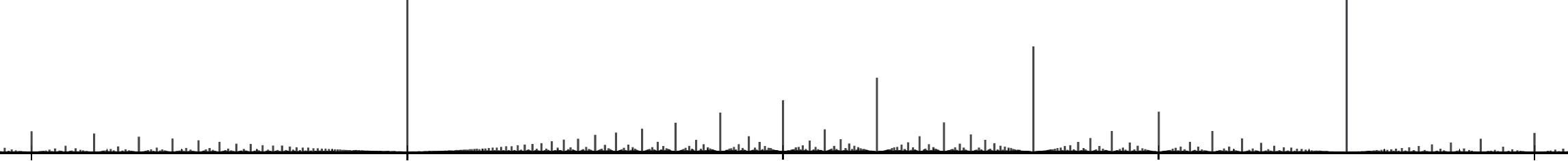}
\end{center}
\caption[A histogram of $\scl$ values for many
words of length $40$ in a free group of rank $2$.]
{A histogram of $\scl$ values for many
words of length $40$ in a free group of rank $2$.  Some of the vertical 
bars are not to scale.}
\label{fig:histogram}
\end{figure}

\section{Surface maps into free groups}
\label{sec:free_groups}

In this section, we show how to parametrize the space of 
admissible surface maps into free groups as a polyhedron 
in a vector space, and in such a way that $\scl$ can be computed 
as a linear function.  Logically, this section follows from
Section~\ref{sec:free_prods}.  However it serves as a warm 
up and introduction to our notation and methods.

\subsection{Notation}

Let $F_k$ be a free group of rank $k$.  We'll denote the rose with $k$ loops by 
$X_k = K(F_k,1)$.  Let $\Gamma \in B_1^H(F_k)$.  We will 
write $\Gamma = \sum_i w_i$, expressing $\Gamma$ as a formal sum of words. 
Note that while $B_1^H(F_k)$ is a vector space over $\R$, it suffices to 
compute $\scl$ over $\Q$ by continuity and over $\Z$ by clearing denominators 
and using homogeneity.
We use the term 
\emph{generator} to mean a generator of $F_k$; in our examples, we'll use 
$a$ and $b$ as generators of $F_k$, and we'll denote inverses with 
capital letters, so $A = a^{-1}$ and $B = b^{-1}$.  The formal sum of words 
$\sum_i w_i$ in these generators represents $\Gamma$, and we will call 
a particular letter in 
a particular location in one of these words a \emph{letter}.  It is important 
to distinguish between a generator and a particular occurrence of that 
generator in the chain $\Gamma$ (i.e., a letter).  Two letters are inverse if the 
generators they denote are inverse.  We use $\Gamma_{i,j}$ to denote letter $j$ 
of word $i$ of $\Gamma$, with indices starting at $0$.  If a chain is written 
out, we will use similar subscripts to reference letters, so if $\Gamma = abAABB+ab$, 
then $a_{0,0}$ denotes the $a$ at index $0$ of word $0$, and $B_{0,4}$ denotes the $B$ at 
index $4$ in the word $0$.

\subsection{Fatgraphs}

\begin{definition}
A \emph{fatgraph} $Y$ is a graph with a cyclic order on the 
incident edges at each vertex.  The cyclic order ensures that a 
fatgraph admits a well-defined \emph{fattening} to a surface $S(Y)$, 
and the surface $S(Y)$ deformation retracts back to $Y$.  
See Figure~\ref{fig:fatgraph_example}.
\end{definition}

Notice that $\chi(Y) = \chi(S(Y))$.

\begin{figure}[ht]
\begin{center}
\includegraphics[scale=0.5]{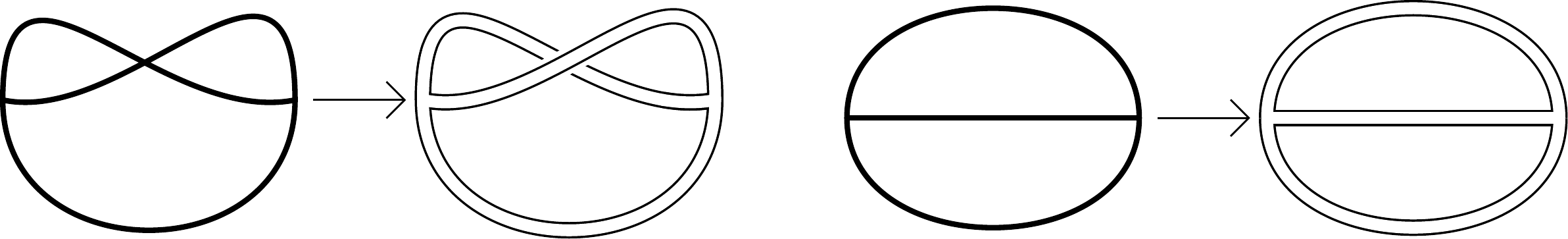}
\caption[Fatgraphs]{Two fatgraphs differing only in the cyclic orders 
at their vertices, 
and their fattenings to a once-punctured torus and a 
trice-punctured sphere.}
\label{fig:fatgraph_example}
\end{center}
\end{figure}

\begin{definition}
A \emph{fatgraph over $F_k$} is a fatgraph with a label on each side of every edge, 
where each label is a generator of $F_k$, and in such a way that inverse generators appear on 
opposite sides of the same edge.  See Figure~\ref{fig:labeled_fatgraph}.
\end{definition}

\begin{figure}[[ht]
\begin{center}
\labellist
\small\hair 2pt
\pinlabel $B$ at 207.517847 5
\pinlabel $b$ at 234.242231 45

\pinlabel $a$ at 373 147
\pinlabel $A$ at 327 131

\pinlabel $B$ at 235 297
\pinlabel $b$ at 224 251

\pinlabel $A$ at 20 234
\pinlabel $a$ at 53 199

\pinlabel $B$ at 87 82
\pinlabel $b$ at 100 38
\endlabellist
\includegraphics[scale=0.35]{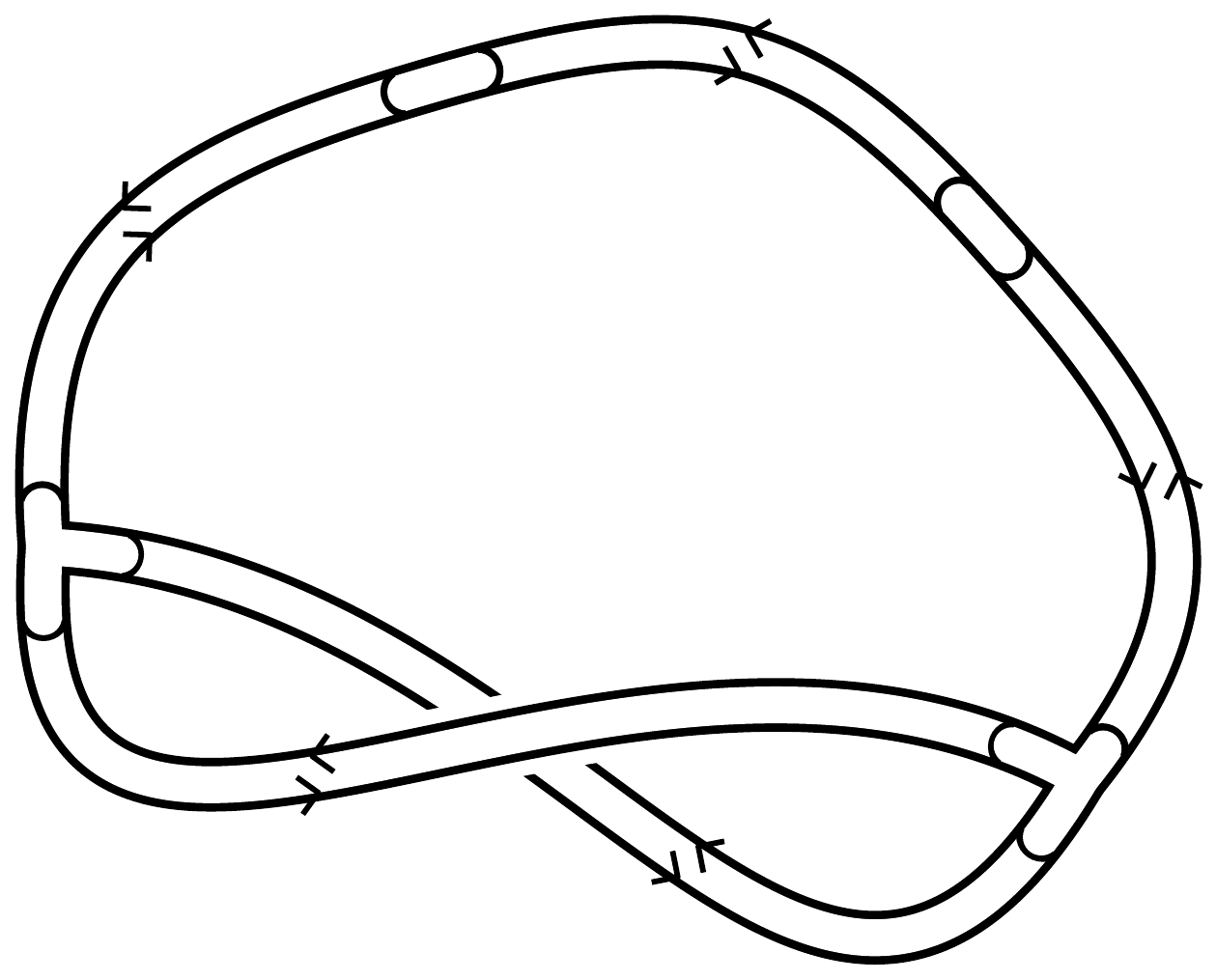}
\caption[A labeled fatgraph]{A labeled fatgraph over $F_2$.  The induced surface 
map takes the boundary to the conjugacy class of the commutator 
$[abAB,BaBA] = abABBaBAbb$.}\label{fig:labeled_fatgraph}
\end{center}
\end{figure}
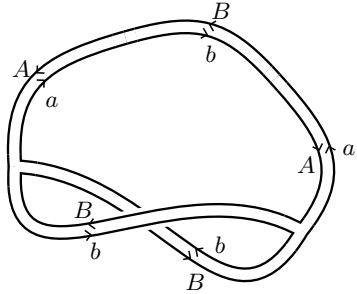

A fatgraph over $F_k$ induces a surface map $S(Y) \to X_k$ by sending 
the vertices to the basepoint of $X_k$ and sending the edges around the 
loops of $X_k$ as instructed by the labels.  There is also 
a converse.
\begin{lemma}[\cite{Culler_surfaces}, Theorem~1.4]
\label{lem:fatgraphs_rep_surface_maps}
Let $S$ be a surface with boundary, and let 
$g: S \to X_k$ be a continuous map.  Then there exists a 
fatgraph $Y$ labeled over $F_k$ with induced map $h:S(Y)\to X_k$ 
such that $-\chi^-(Y) \le -\chi^-(S)$ and $g(\partial S) = h(\partial S(Y))$.
\end{lemma}

The fatgraph $Y$ is obtained from the surface $S$ by first compressing 
it and then deforming it to give the fatgraph structure.

\subsection{Surface maps into free groups}
\label{sec:surface_maps_into_free_groups}

We now look more closely at labeled fatgraphs in order to decompose them into 
\emph{rectangles} and \emph{triangles}.  This is essentially the same 
idea as the \texttt{scallop} algorithm, except we decompose further to 
improve the rigor and computational complexity.
The tedious 
notation is important only for bookkeeping reasons --- the idea of the 
decomposition is quite straightforward and is contained in the figures.

Let $S(Y) \to X_k$ be a labeled fatgraph map which is admissible for the 
chain $\Gamma$.  Consult Figure~\ref{fig:fatgraph}, 
and consider decomposing $S(Y)$ into pieces: \emph{rectangles} (the edges of $S(Y)$), and 
\emph{polygons} (the vertices of $S(Y)$). 

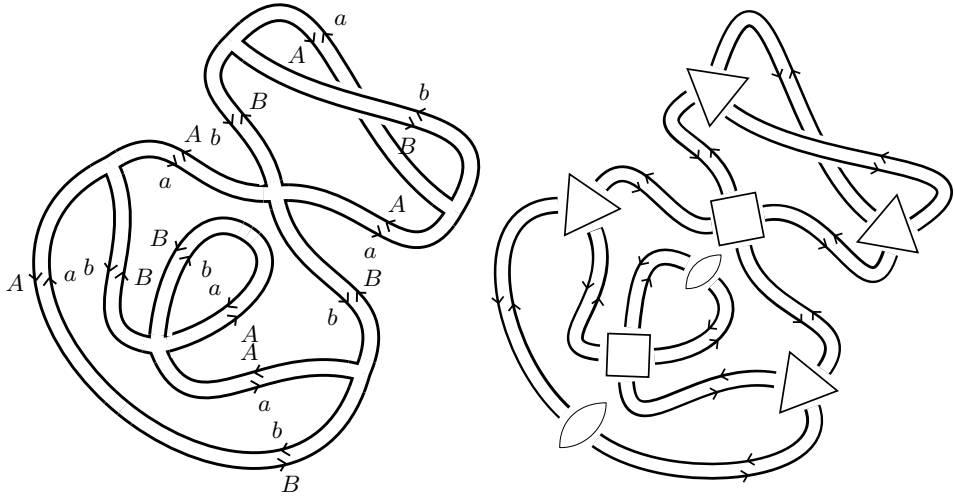
\begin{figure}[ht]
\begin{center}
\labellist
\small\hair 2pt
\pinlabel $a$ at 278.434188 397.221402
\pinlabel $A$ at 240.808208 367.417281
\pinlabel $A$ at 326.668456 239.761059
\pinlabel $a$ at 303.939909 197.483250
\pinlabel $b$ at 349.865943 335.551093
\pinlabel $B$ at 335.360039 289.795451
\pinlabel $a$ at 128.480549 258.093056
\pinlabel $A$ at 152.180918 299.833833
\pinlabel $A$ at -0.469690 172.086956
\pinlabel $a$ at 47.094105 178.543379
\pinlabel $b$ at 62.682862 185.883850
\pinlabel $B$ at 109.956784 177.566621
\pinlabel $b$ at 171.382889 297.080053
\pinlabel $B$ at 208.202230 327.875119
\pinlabel $b$ at 271.341896 143.209135
\pinlabel $B$ at 305.839905 176.584086
\pinlabel $A$ at 200.180582 112.725405
\pinlabel $a$ at 213.480491 66.604776
\pinlabel $b$ at 224.016726 46.102705
\pinlabel $B$ at 235.617923 -0.474244
\pinlabel $A$ at 200.574952 128.733333
\pinlabel $a$ at 170.830782 166.406724
\pinlabel $b$ at 163.641698 184.377966
\pinlabel $B$ at 123.303345 210.393678
\endlabellist
\includegraphics[scale=0.45]{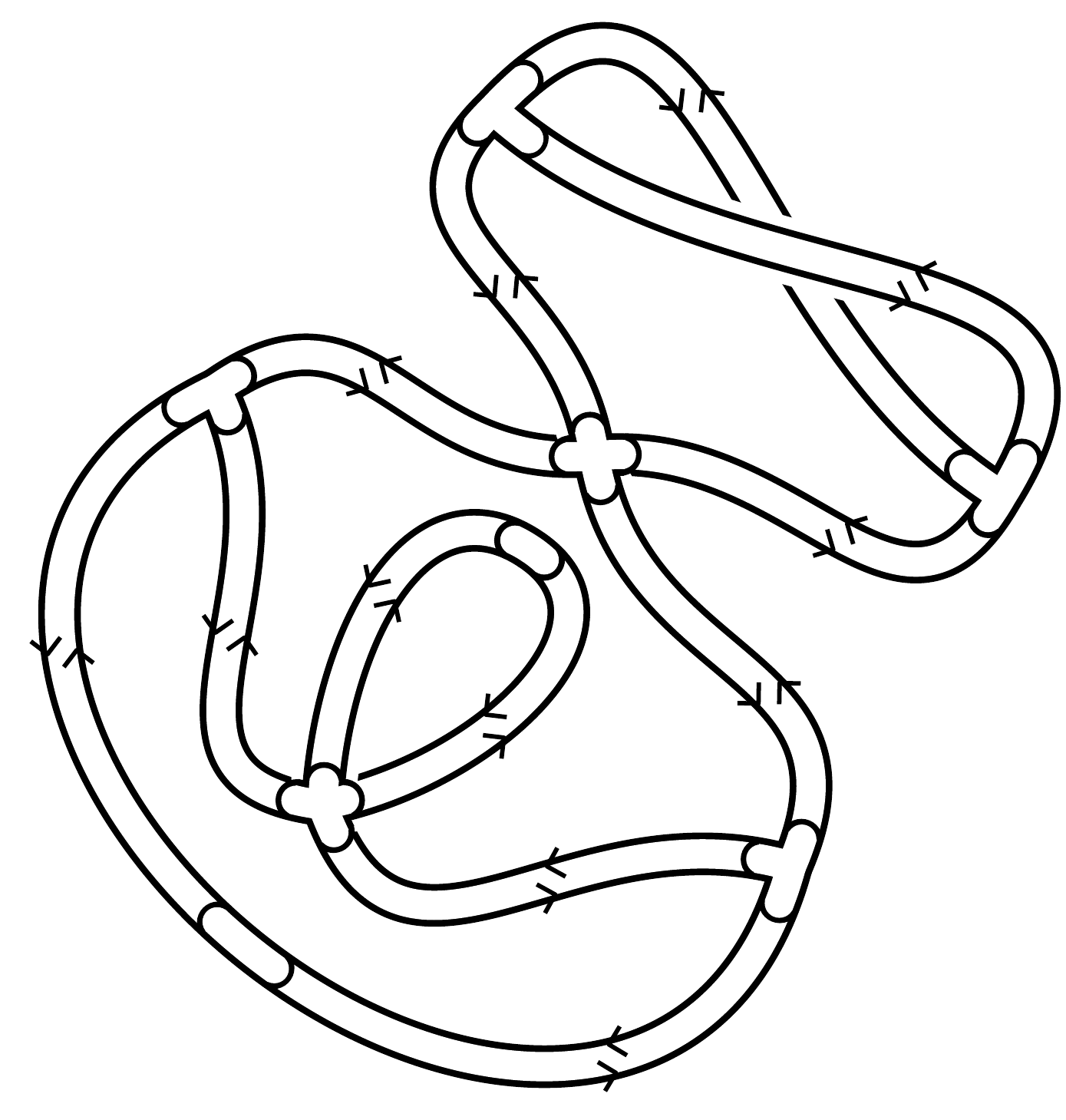}
\includegraphics[scale=0.4]{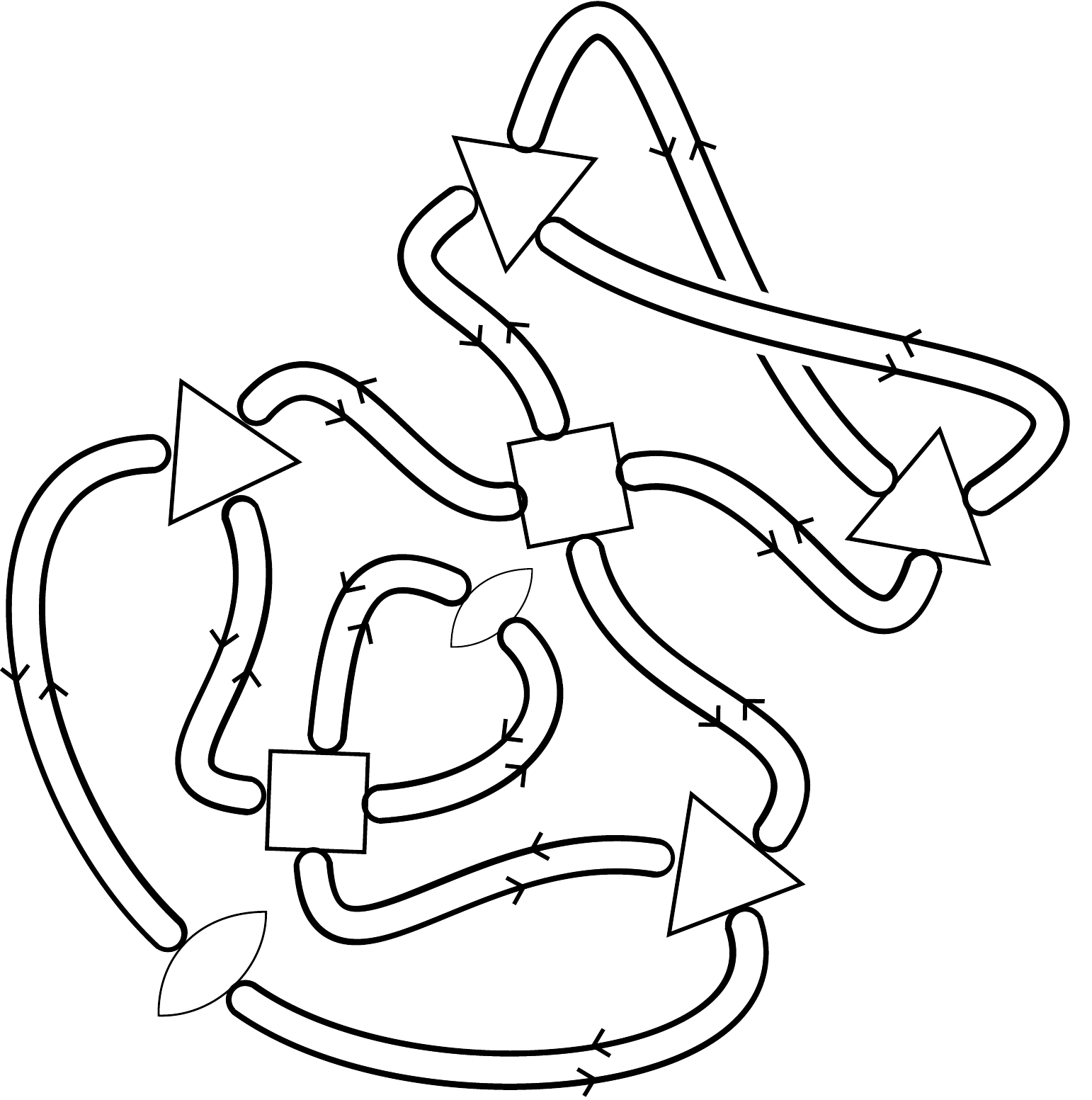}
\end{center}
\caption[A fatgraph structure]{A fatgraph structure on an extremal 
(in particular, admissible) surface for the chain $abAABB+ab$, 
and the same fatgraph split into rectangles and polygons.}
\label{fig:fatgraph}
\end{figure}

On each rectangle in $S$, we find two letters in $\Gamma$, 
say $x$ and $y$, which must be 
inverses, and we denote the rectangle with these long sides labeled by 
$x$ and $y$ by $r(x,y)$.  The rectangle $r(x,y)$ is the same as $r(y,x)$.  
We need to record the interface between rectangles and the fatgraph vertices, which 
happens along the short \emph{edges} of the rectangles.  Each edge is determined
by the adjacent long labeled sides, and we'll denote an edge as an ordered pair $e(x,y)$ 
of the incoming labeled side, say $x$, followed by the outgoing side $y$.  
Note that $e(x,y) \ne e(y,x)$.  Formally, the set of all edges is all pairs of 
letters in $\Gamma$.  Reading counterclockwise 
around a rectangle $r(x,y)$, we find the side labeled $x$, 
then the edge $e(x,y)$, then $y$, then $e(y,x)$.

Each vertex of the fatgraph, after cutting off the rectangles, 
becomes a polygon whose sides are all edges (which were attached 
to rectangle edges), and we 
want to record which edges these are.  Note that 
each vertex of the polygon lies between two letters in the boundary.  
The name of an edge of a polygon is $e(x,y)$, where 
$x$ is the letter \emph{incoming} to the initial vertex of the edge, 
and $y$ is the letter \emph{outgoing} from the terminal vertex of 
the edge.  

When we cut off a rectangle, we produce an edge on the polygon, and 
an edge on the rectangle.  
If the rectangle edge is $e(\Gamma_{i,j},\Gamma_{k,l})$, 
then the attaching polygon edge is $e(\Gamma_{k,l-1}, \Gamma_{i,j+1})$. 
See Figures~\ref{fig:rectangle} and~\ref{fig:polygon}.

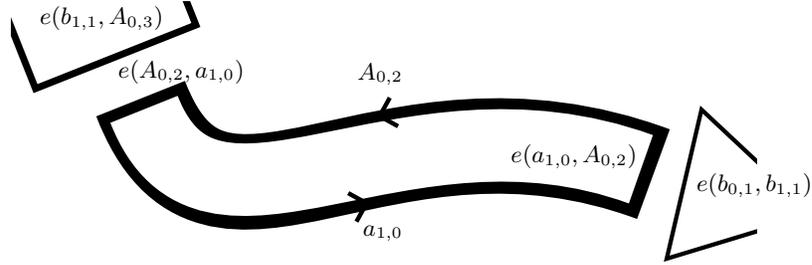
\begin{figure}[ht]
\begin{center}
\labellist
\small\hair2pt
\pinlabel $a_{1,0}$ at 140 16
\pinlabel $A_{0,2}$ at 139 76
\pinlabel ${e(a_{1,0}, A_{0,2})}$ at 212 45
\pinlabel ${e(A_{0,2}, a_{1,0})}$ at 64 77
\pinlabel ${e(b_{0,1}, b_{1,1})}$ at 280 34
\pinlabel ${e(b_{1,1}, A_{0,3})}$ at 34 97
\endlabellist
\includegraphics{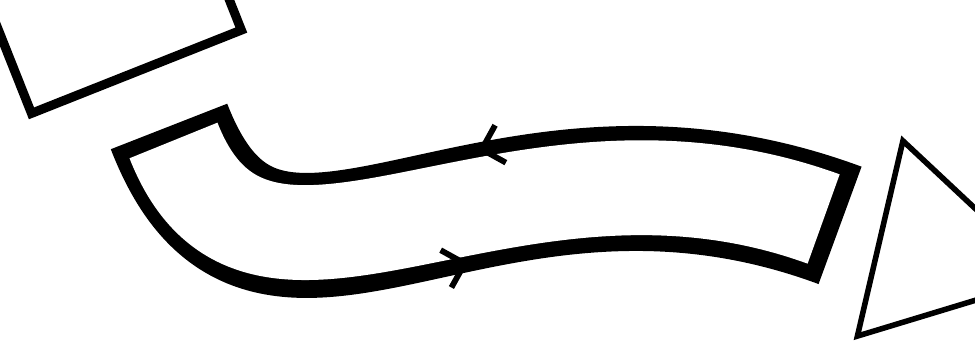}
\end{center}
\caption[A rectangle]{The rectangle $r(a_{1,0},A_{0,2})$ for the chain $abAABB+ab$, 
pictured as an enlarged piece of Figure~\ref{fig:fatgraph}, 
as found towards the lower right of Figure~\ref{fig:fatgraph}, 
with sides and short edges labeled.}
\label{fig:rectangle}
\end{figure}

Now consider what edges we can find around a polygon.  If 
two rectangles are adjacent, they must correctly read off 
a portion of the chain $\Gamma$.  In other words, the 
incoming letter on a rectangle must be immediately before 
(in $\Gamma$) the outgoing letter of the next (counterclockwise) 
rectangle.  With the definitions above, we find that 
if edge $e_2$ follows $e_1$ counterclockwise on the boundary 
of a polygon, then if $e_1 = e(x, \Gamma_{i,j})$, 
where $x$ is any letter in $\Gamma$, then  
then $e_2$ must be of the form $e(\Gamma_{i,j-1}, y)$ 
for some letter $y$.  This may be counterintuitive, as the 
\emph{following} edge is labeled by the \emph{previous} letter.  
Consult Figure~\ref{fig:polygon}.

\begin{figure}[ht]
\labellist
\small\hair2pt
\pinlabel ${e(b_{0,1},b_{1,1})}$ at 153 185
\pinlabel ${e(a_{0,1},A_{0,2})}$ at 62 210
\pinlabel $a_{0,1}$ at 13 120
\pinlabel $A_{0,2}$ at 130 313

\pinlabel ${e(a_{0,1}, B_{0,5})}$ at 165 112
\pinlabel ${e(B_{0,4}, b_{1,1})}$ at 166 70
\pinlabel $B_{0,4}$ at 270 30
\pinlabel $b_{1,1}$ at 30 47

\pinlabel ${e(B_{0,4}, A_{0,2})}$ at 184 148
\pinlabel ${e(b_{0,1}, B_{0,5})}$ at 260 200
\pinlabel $b_{0,1}$ at 233 319
\pinlabel $B_{0,5}$ at 310 97
\endlabellist
\centering
\includegraphics[scale=0.65]{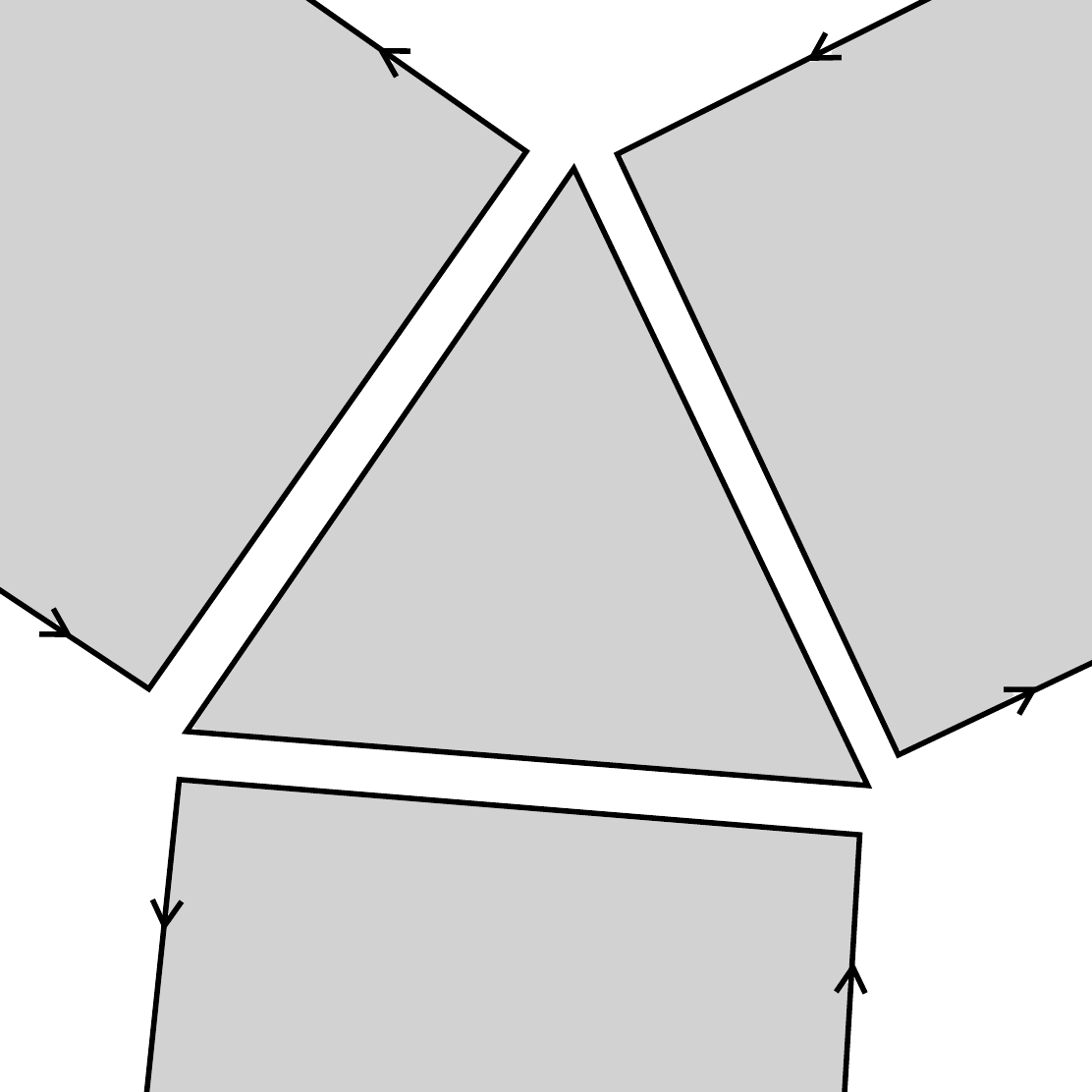}
\caption[A polygon and incident rectangles]
{A polygon and the rectangles incident to it, with edges labeled.  
This example is the lower-right hand triangle 
in Figure~\ref{fig:fatgraph}.}
\label{fig:polygon}
\end{figure}
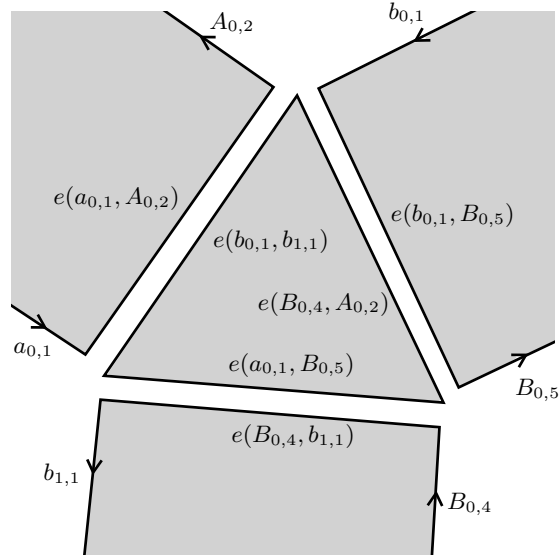

There are only finitely many rectangles which can possibly appear in 
a fatgraph admissible for $\Gamma$, since each rectangle must correspond 
to a pair of letters in $\Gamma$ which are inverses of each other.  
However, there are infinitely many types of polygons which could occur, 
because a polygon can have an arbitrary number of sides.  
To break the fatgraph into finitely 
many types of pieces, we need to cut up the polygons
into \emph{triangles}, which we can always do; see Figure~\ref{fig:triangles}.  
It is important to cut the polygons into triangles with 
no internal vertices.

\begin{figure}[ht]
\begin{center}
\includegraphics[scale=0.5]{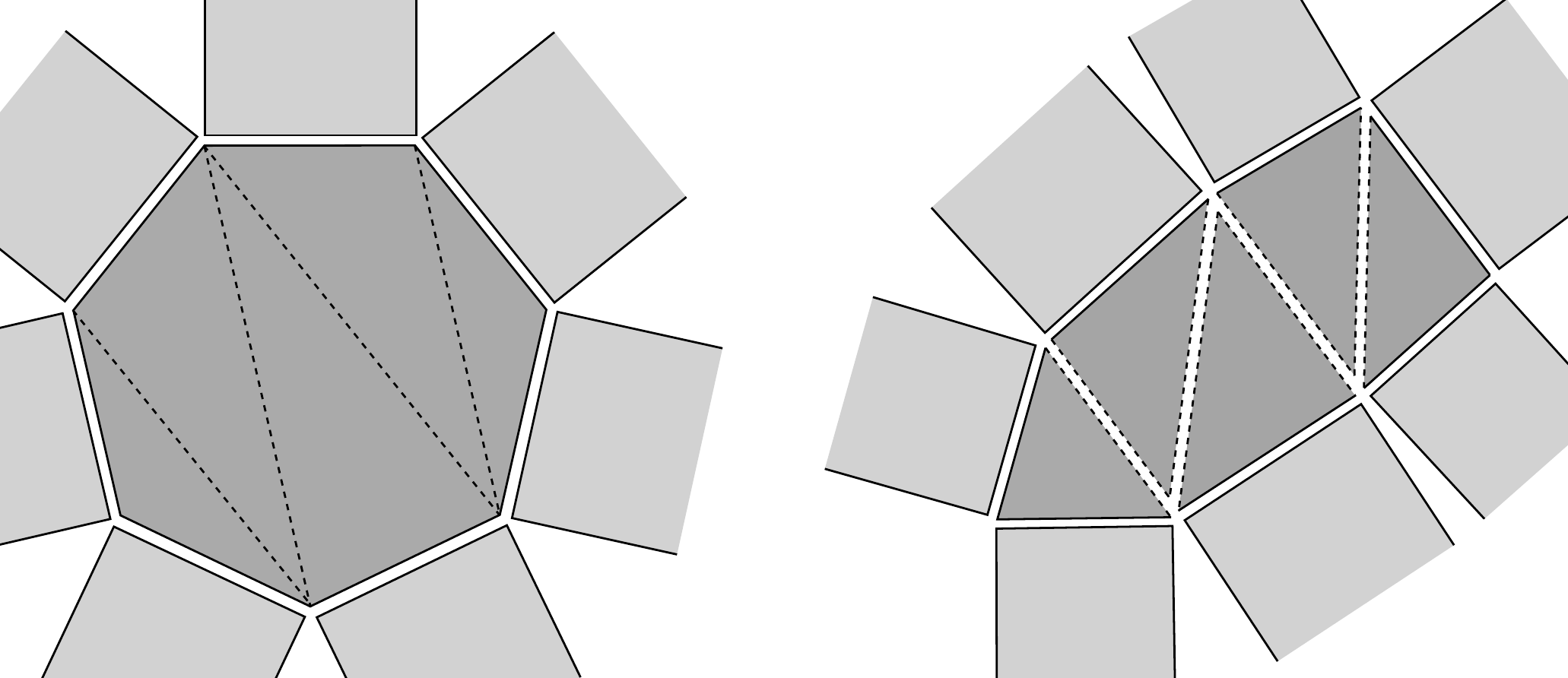}
\end{center}
\caption[Cutting a polygon into triangles]
{Cutting a polygon into triangles.}
\label{fig:triangles}
\end{figure}

The boundary of each triangle is three edges.  Some of these 
are inherited from the boundary of the polygon that we cut, and 
these edges get the labels they had originally.  There are also new 
edges that arise from cutting the polygon.  Actually, these can 
be labeled in the same way: each vertex of a triangle 
lies at some vertex of a polygon, and we label 
an edge of a triangle as $e(x,y)$ when $x$ is the letter incoming to the initial 
vertex, and $y$ is the letter outgoing from the terminal vertex.  
Just like with the attachment between the edges on 
rectangles and edges on polygons, if we find edge 
$e(\Gamma_{i,j}, \Gamma_{k,l})$ on a triangle, then the 
edge on the triangle on the other side of the cut will be 
$e(\Gamma_{k,l-1},\Gamma_{i,j+1})$.  We denote a triangle 
$t(e_1, e_2, e_3)$, using its cyclically ordered edges.  
See Figure~\ref{fig:triangle_labels}.
Note that each type of rectangle, triangle, and edge may appear many times 
in our decomposition.

\begin{figure}[ht]
\begin{center}
\labellist
\small\hair2pt
\pinlabel $a_{0,0}$ at 322 265
\pinlabel $A_{0,3}$ at 167 322
\tiny
\pinlabel $e(A_{0,3},a_{0,0})$ at 235 282
\pinlabel $e(B_{0,5},B_{0,4})$ at 230 230

\small
\pinlabel $B_{0,4}$ at 147 313
\pinlabel $b_{1,1}$ at -10 175
\tiny
\pinlabel $e(b_{1,1},B_{0,4})$ at 68 250
\pinlabel $e(A_{0,3},a_{1,0})$ at 100 190

\small
\pinlabel $a_{1,0}$ at -10 145
\pinlabel $A_{0,2}$ at 137 0
\tiny
\pinlabel $e(b_{1,1},A_{0,3})$ at 100 130
\pinlabel $e(A_{0,2},a_{1,0})$ at 50 90

\small
\pinlabel $A_{0,3}$ at 172 -6
\pinlabel $a_{1,0}$ at 323 50
\tiny
\pinlabel $e(a_{1,0},A_{0,3})$ at 235 35

\small
\pinlabel $b_{1,1}$ at 325 75
\pinlabel $B_{0,5}$ at 325 240
\tiny
\pinlabel $e(B_{0,5},b_{1,1})$ at 335 157

\pinlabel $(C)$ at 148 160
\pinlabel $(D)$ at 182 150
\pinlabel $(E)$ at 225 157
\pinlabel $(F)$ at 250 140
\pinlabel $(G)$ at 237 75
\pinlabel $(H)$ at 275 140
\normalsize
\pinlabel $(C):\quad~e(A_{0,2},B_{0,4})$ at 440 300
\pinlabel $(D):\quad~e(A_{0,3},A_{0,3})$ at 440 275
\pinlabel $(E):\quad~e(A_{0,2},a_{0,0})$ at 440 250
\pinlabel $(F):\quad~e(B_{0,5},A_{0,3})$ at 440 225
\pinlabel $(G):\quad~e(A_{0,2},b_{1,1})$ at 440 200
\pinlabel $(H):\quad~e(a_{1,0},a_{0,0})$ at 440 175
\endlabellist
\hspace{-3cm}\includegraphics[scale=0.6]{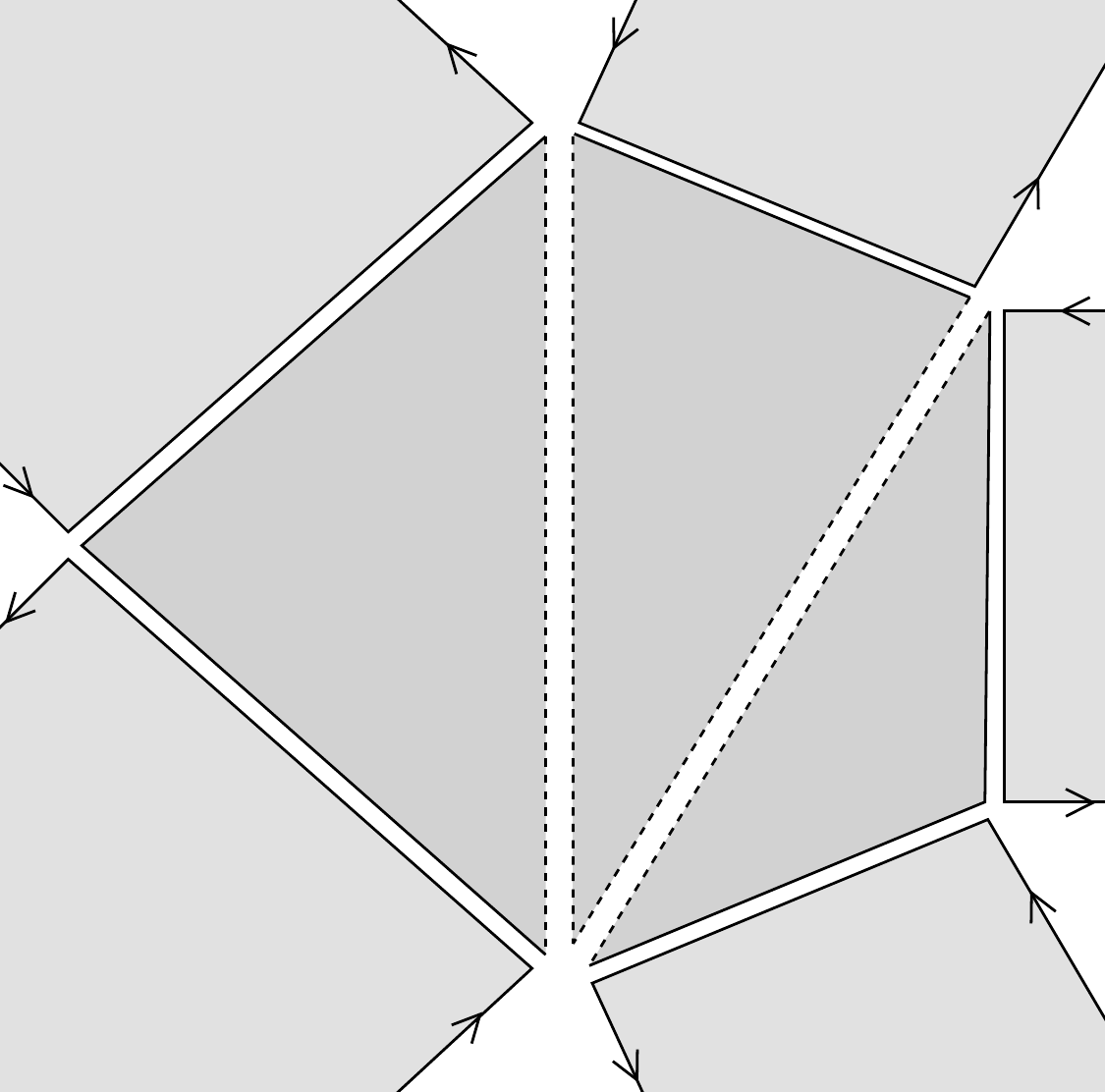}
\end{center}
\caption[Triangle edges]{Labeling the edges of the triangles 
for a hypothetical collection of triangles and rectangles 
for the chain $abAABB+ab$.  For clarity, some of the labels have been 
moved off to the right, as indicated.}
\label{fig:triangle_labels}
\end{figure}
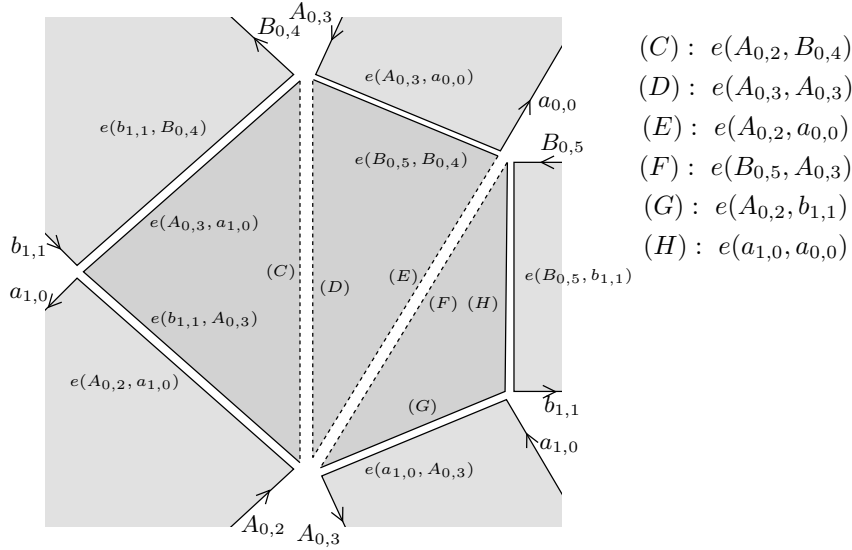

\subsection{Building surface maps}

To summarize, we have shown that, after possibly compressing, 
any surface map admissible for a chain $\Gamma$ is carried by a 
labeled fatgraph map, and 
this fatgraph can be cut into rectangles, which are specified by 
a pair of inverse letters in $\Gamma$, and triangles, which 
are specified by three cyclically ordered compatible edges.  
Recall that three edges are compatible if each edge of the 
form $e(x, \Gamma_{i,j})$ for some letter $x$ is cyclically followed
by an edge of the form $e(\Gamma_{i,j-1}, y)$ for some letter $y$.
To obtain the fatgraph back from the pieces into which we cut it, 
we glue back rectangles and triangles along edges.  
Each edge appears in possibly many rectangles and triangles, but if 
we remember the original fatgraph, we can glue appropriately.  
All gluings glue an edge labeled $e(\Gamma_{i,j},\Gamma_{k,l})$ to one 
labeled $e(\Gamma_{k,l-1}, \Gamma_{i,j+1})$.  We'd like to show that if 
instead of starting with the fatgraph and cutting, we 
start with a collection of the pieces satisfying some constraints, then 
we can glue them to give a labeled fatgraph and thus an admissible 
surface map for $\Gamma$.  It turns out that we can't always assemble the 
pieces into a fatgraph, but we \emph{can} assemble them into a surface map, 
and that will be enough.

We must be slightly careful to avoid degenerate pieces.  A 
\emph{dummy edge} is an edge $e(x,y)$, where $y$ 
follows $x$ cyclically in $\Gamma$.  Since $\Gamma$ 
is reduced, no rectangle can contain a dummy edge, but we make the 
additional constraint that no triangle can contain a dummy edge.  
Note that any surface admissible for $\Gamma$ 
that is glued along a dummy edge 
can be cut along the dummy edge to produce a surface admissible 
for $\Gamma$ and with smaller Euler characteristic.

We define a \emph{piece} to be either a rectangle or triangle.  Let $\PP_\Gamma$ 
be the set of all types of pieces.  Note that $|\PP_\Gamma| \le |\Gamma|^3$, since 
the rectangles are specified by two letters, and the triangles by three.
In particular, $\PP_\Gamma$ is finite.  Let $\EE_\Gamma$ be the set of all 
edges, not including dummy edges, 
and similarly, $|\EE_\Gamma| \le |\Gamma|^2$.  Let $V_\Gamma = \Q[\PP_\Gamma]$ be the 
vector space spanned by the pieces, and let $E_\Gamma = \Q[\EE_\Gamma]$ be the 
vector space spanned by the edges.  There is a boundary 
map $\partial : V_\Gamma \to E_\Gamma$ defined on generators by taking 
each rectangle or triangle to the sum of its edges.  Specifically, 
$\partial(r(x,y)) = e(x,y) + e(y,x)$, and 
$\partial(t(e_1, e_2, e_3)) = e_1 + e_2 + e_3$.  On edges, there is a 
map which records if a collection of edges is compatible: we define 
$\iota:E_\Gamma \to E_\Gamma$ on generators by 
$\iota( e(\Gamma_{i,j},\Gamma_{k,l})) = -e(\Gamma_{k,l-1}, \Gamma_{i,j+1})$.
These maps will allow us to determine whether a collection of pieces can 
be glued up to produce a fatgraph.

We also need to extract Euler characteristic from the pieces.  We define 
$\chi:V_\Gamma \to \Q$ on generators to be $0$ on every rectangle, and $-1/2$ on 
every triangle.  For each word $w_i$ in the sum $\Gamma = \sum w_i$, we define 
$N_i:V_\Gamma \to \Q$ to be zero on all triangles, and $1$ on a rectangle 
if and only if one of the sides of the rectangle is the 
first letter in $w_i$.  Note that the first letter of $w_i$ cannot 
be on both sides of a rectangle, because then it would be its own inverse.

The set of positive vectors in the subspace $\ker (\iota \circ \partial)$ 
is a cone in $V_\Gamma$, which we denote by $C_\Gamma$, and the intersection 
of $C_\Gamma$ with the affine subspace $\{v \in V\,|\, N_i(v) = 1 \, \forall\,i\}$ 
is a polyhedron, which we denote by $P_\Gamma$.  This is the 
\emph{admissible polyhedron} for $\Gamma$.

\begin{proposition}\label{prop:free_groups}
In the above notation, $\scl(\Gamma) = \inf_{v \in P_\Gamma} -\chi(v)/2$.  
Furthermore, an extremal surface for $\Gamma$ can be extracted 
from a minimizing vector in $P_\Gamma$.
\end{proposition}

The proof of Proposition~\ref{prop:free_groups} breaks into 
two lemmas, one for each direction of an inequality.

\begin{lemma}\label{lem:free_groups_1}
Given $f:S \to X_k$ an admissible surface for $\Gamma$, there is 
a vector $v \in P_\Gamma$ so that $-\chi(v)/2 \le -\chi^-(S)/2n(S,f)$.
\end{lemma}
\begin{proof}
By Lemma~\ref{lem:fatgraphs_rep_surface_maps}, there is a 
labeled fatgraph $Y$ with $-\chi(Y) \le -\chi(S)$ so that 
the map $f':S(Y) \to X_k$ has the same boundary image as $S$.  
In particular, $(S(Y), f')$ is admissible for $\Gamma$, and 
$n(S(Y), f') = n(S,f)$.  Cut $Y$ into rectangles and triangles, 
and record the number of each type in an (integral) vector $v \in V_\Gamma$.  
The total degree of $S(Y)$ over each loop $w_i$ can be determined 
by counting the number of times that $w_i$ appears in the 
boundary of $S(Y)$, which is the same as counting the number of times 
that the first letter of $w_i$ appears.  That is, 
$N_i(v) = n(S(Y),f') = n(S,f)$ for all $i$.  Furthermore, 
$\chi(v) = \chi^-(S(Y))$, because $S(Y)$ is homotopy equivalent to 
its spine, which is a graph with one $2$-valent vertex for each rectangle 
and one $3$-valent vertex for each triangle, so the definition of 
$\chi(v)$ on generators clearly computes Euler characteristic. 
Scaling $v$ by $1/N_i(v)$ (for any $i$, as they are all equal) 
gives a vector in $P_\Gamma$, and we have
\[
-\frac{1}{2}\chi\left( \frac{1}{N_i(v)}v \right)  = -\chi(v) \frac{1}{2N_i(v)} 
                                  = \frac{-\chi^-(S(Y))}{2n(S(Y),f')}
                                  \le \frac{-\chi^-(S)}{2n(S,f)},
\]
which completes the proof.
\end{proof}

\begin{lemma}\label{lem:free_groups_2}
Given $v \in P_\Gamma$, there is a surface map $f: S \to X_k$ admissible 
for $\Gamma$ with $-\chi^-(S)/2n(S,f) \le -\chi(v)/2$.
\end{lemma}
\begin{proof}
There is some $k \in \Z$ so that $kv$ is integral.  Therefore, $kv$
represents a collection of pieces with the property that 
every edge $e$ appears the same number of times as its 
gluing partner $\iota(e)$.  Each piece may appear many times, and each edge 
many times in many pieces.  Glue the pieces arbitrarily along $\iota$-pairs of edges to produce a surface $S$.  
Define a map $f:S \to X_k$ which takes all the triangles to the basepoint 
and all the rectangles around the appropriate edges.  Notice that 
this map may have branch points, since it is possible that we 
glued up the triangles in such a way that not every vertex is on the boundary 
of $S$ (the triangles may tile a polygon with a central vertex, for example).  
It is also possible that this surface is compressible (if the triangles 
tile an annulus which is crushed to the basepoint, for example).  

Regardless of the branch points or compressibility of the surface map, 
it is true that $-\chi^-(S) \le -\chi(kv)$, since we have only added branch 
points, and thus only decreased $-\chi^-$.  If the reader likes, we can subtract 
the preimages of the branch points from $S$, which only increases $-\chi$.  The surface 
now deformation retracts to 
its spine, which is the same graph as above, with $2$ and $3$ valent vertices for 
the rectangles and triangles, so its Euler characteristic is computed by the
linear function $\chi$ on $V_\Gamma$.

Also, we have $n(S,f) = N_i(kv) = kN_i(v) = k$ (for all $i$), by the same 
counting-first-letters argument as above, so we compute
\[
\frac{-\chi^-(S)}{2n(S,f)} \le \frac{-\chi(kv)}{2N_i(kv)} = -\frac{1}{2}\chi(v),
\]
which completes the proof.
\end{proof}

\begin{proof}[Proof of Proposition~\ref{prop:free_groups}]
To prove Proposition~\ref{prop:free_groups}, we simply 
apply Lemmas~\ref{lem:free_groups_1} and~\ref{lem:free_groups_2}, which 
immediately give
\[
\scl(\Gamma) = \inf_{(S,f)} \frac{-\chi^-(S)}{2n(S,f)} =  \inf_{v \in P_\Gamma} -\chi(v)/2.
\]
It remains to show the existence of an extremal surface.  
Given a minimizing vector $v \in P_\Gamma$, we use Lemma \ref{lem:free_groups_2} 
to construct an admissible surface map $f: S \to X_k$ with 
$-\chi^-(S)/2n(S,f) \le -\chi(v)/2$.  A priori, the inequality may be 
strict.  However, since $v$ is minimal, it must be an equality.  
Notice that $S$ must therefore contain no branch points, and in particular 
must be an honest fatgraph.
\end{proof}

\section{Surface maps into free products}
\label{sec:free_prods}

\subsection{Introduction}

In this section, we extend the construction of Section~\ref{sec:free_groups} 
to handle surface maps into free products of cyclic groups.  
Let $G= *_j G_j$, where $G_j$ is cyclic.  Let $o_j$ be the order of $G_j$, 
where $o_j=0$ if $G_j$ is infinite.  Let $\Gamma \in B_1^H(G)$.  
We can write $\Gamma = \sum w_i$, where $w_i \in G$.  Since $G$ is not 
free, there may be many ways of writing each word $w_i$.  To simplify our argument, 
we cyclically rewrite each word $w_i$ so that each generator of a finite factor appears 
only with a positive power and so that $w_i$ is as short as possible.  This form is 
(cyclically) unique.  It is possible that some of the $w_i$ are contained 
in a single factor $G_j$; that is, they are powers of the generators.  
If a word $w_i$ is a power of a generator in a finite factor, 
we call it a \emph{finite abelian loop}.  They would complicate the search 
for a surface, but fortunately, we can ignore them, 
because in $B_1^H(G)$, if we let $c$ be the product of the 
finite orders $o_j$, then $\Gamma = c\Gamma/c$, and 
every finite abelian loop in $c\Gamma$ is trivial and can be removed. 
Consequently, we have the following observation, which we record as a 
lemma.

\begin{lemma}\label{lem:ignore_abelian_loops}
Let $\Gamma'$ be $\Gamma$ with the finite abelian loops removed.  
Then $\Gamma = \Gamma'$ in $B_1^H(G)$.  Therefore, 
$\scl(\Gamma') = \scl(\Gamma)$, and an extremal surface 
for $\Gamma$ can be produced from an extremal surface for $\Gamma'$.
\end{lemma}

As a result of Lemma~\ref{lem:ignore_abelian_loops}, we will always 
assume that $\Gamma$ has no finite abelian loops, although we will remark 
where this becomes important.  
We let $X_G$ be a $K(G,1)$ for our group.  For concreteness, 
set $X_G$ to be the standard presentation complex for $G$; that is, 
a rose with $2$ cells glued to powers of the generators of 
the finite factors.  We are going to build surface maps into $X_G$ 
by gluing together pieces as before.  However, we don't have 
Lemma~\ref{lem:fatgraphs_rep_surface_maps} on which to fall back 
to give us a nice combinatorial structure, so we need to 
build it from scratch, while incorporating the finite factors.  

In order to decompose surface maps, we need to introduce 
new combinatorial pieces.  Previously, we were given 
a surface map, we cut it into pieces, and we 
recorded the kinds of pieces we got.  
For this section, the decomposition is not so trivial, so 
we will first define what the pieces are and what the polyhedral 
structure is, and then we will go back and prove that it 
parametrizes surface maps into $X_G$.

Recall that an \emph{edge} is an ordered pair $e(x,y)$ of letters in $\Gamma$.  
Previously, we did not allow dummy edges; that is, those edges $e(x,y)$ 
such that $y$ cyclically follows $x$ in $\Gamma$.  For this section, 
we do allow them on the \emph{group teeth} defined below.  
However, the linear programming ignores them, and they remain not 
allowed in triangles and rectangles.

\subsection{Group polygons}

The main new combinatorial pieces are \emph{group polygons}.
\begin{figure}[ht]
\begin{center}
\labellist
\small
 \pinlabel {$a_{0,1}$} at 64 1
 \pinlabel {$e(a_{0,1},a_{0,0})$} at 122 3
 \pinlabel {$a_{0,0}$} at 120 40
 \pinlabel {$e(a_{0,0},a_{0,3})$} at 148 82
 \pinlabel {$a_{0,3}$} at 98 104
 \pinlabel {$e(a_{0,3},a_{0,1})$} at 63 128
 \pinlabel {$a_{0,1}$} at 28 104
 \pinlabel {$e(a_{0,1},a_{0,4})$} at -20 81
 \pinlabel {$a_{0,4}$} at 6 40
 \pinlabel {$e(a_{0,4},a_{0,1})$} at 4 2
\endlabellist
\includegraphics[scale=1.1]{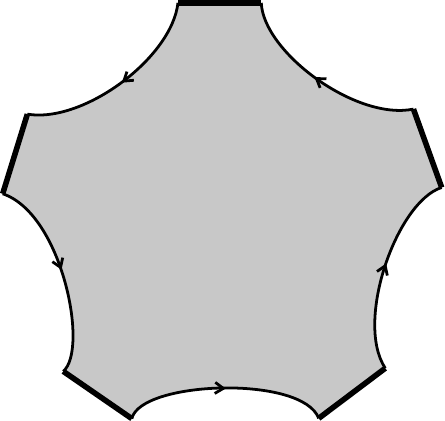}
\end{center}
\caption[A group polygon]
{A group polygon for the chain $\Gamma = aabaaB$ in the 
group $\Z/5\Z * \Z$ generated by $a$ and $b$.}
\label{fig:scylla:group_poly}
\end{figure}
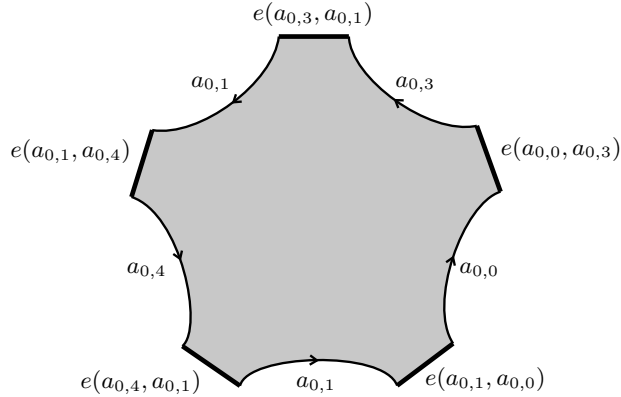
A group polygon 
is associated with a finite free factor $G_j = \Z/o_j\Z$, and it 
has two kinds of sides, which alternate: labeled sides, each 
labeled by a letter in $\Gamma$ 
representing the generator of $G_j$, and \emph{edges}, 
which will serve the same purpose (being glued) as 
edges as in Section~\ref{sec:free_groups}.  
The edge on a group polygon between the labeled sides $x$ and $y$ is 
the edge $e(x,y)$.  Notice that dummy edges \emph{can} 
occur on a group polygon, and we will need to be careful 
about this later when we compute Euler characteristic.
A group polygon in $G_j$ contains exactly $o_j$ edges and $o_j$ labeled sides.  See 
Figure~\ref{fig:scylla:group_poly}.  Formally, 
a group polygon is a cyclic $2o_j$-tuple recording the sides and edges:
$(x_0, e(x_0, x_1), x_1, \ldots, x_{o_j-1}, e(x_{o_j-1}, x_0))$, 
where each $x_i$ is a letter in $\Gamma$ in $G_j$.

\subsection{Group teeth}

In order to imitate Section~\ref{sec:free_groups}, we should build 
surfaces out of polynomially-many types of pieces.  Unfortunately, 
since each group polygon has $o_j$ side labels,
the number of group polygons is exponential in the ranks of the finite factors, 
so we must break the group polygons into smaller pieces.  These 
are \emph{group teeth}.  See Figure~\ref{fig:group_teeth}.

A group tooth is associated with a finite free factor $G_j = \Z/o_j\Z$, and 
it is defined as a $4$-tuple $gt(x,y,n,z)$, where $x,y,z$ are letters in 
$\Gamma$ in the factor $G_j$ (i.e. $x,y,z$ are instances of the 
generator of the $j$th factor), and $n < o_j$.  
We require that if $n=0$, then $x=z$, and if $n=o_j-1$, then $y=z$.
The two labeled sides of a group tooth are labeled 
by $x$ and $y$, and the middle edge, which is an edge as above in the 
sense that it will be glued to other pieces, is $e(x,y)$.  We say that 
the group tooth is \emph{based at $z$}.  The name comes from the fact 
that a group tooth looks like a tooth on a bicycle sprocket.

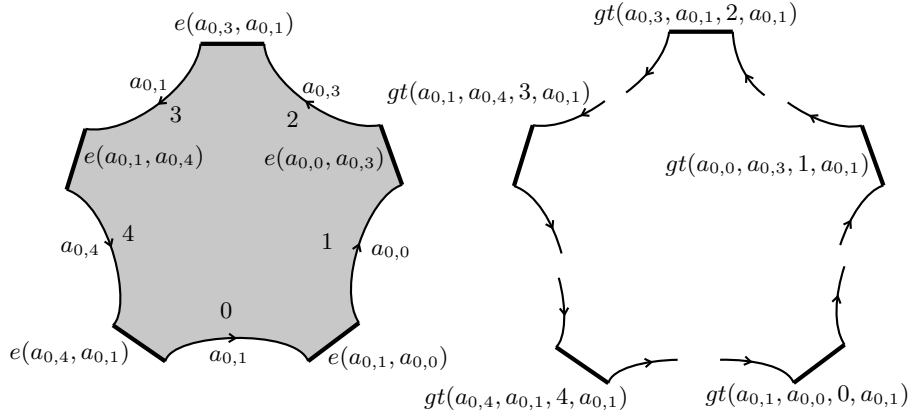
\begin{figure}[ht]
\begin{center}
\labellist
\small
 \pinlabel {$a_{0,1}$} at 62 10
 \pinlabel {$e(a_{0,1},a_{0,0})$} at 122 10
 \pinlabel {$a_{0,0}$} at 123 50
 \pinlabel {$e(a_{0,0},a_{0,3})$} at 98 85
 \pinlabel {$a_{0,3}$} at 97 111
 \pinlabel {$e(a_{0,3},a_{0,1})$} at 64 136                                                                                                        
 \pinlabel {$a_{0,1}$} at 32 113                                                                                                              
 \pinlabel {$e(a_{0,1},a_{0,4})$} at 30 86                                                                                                                
 \pinlabel {$a_{0,4}$} at 6 51                                                                                                                
 \pinlabel {$e(a_{0,4},a_{0,1})$} at 2 12                                                                                                               
 \pinlabel {$0$} at 61 28                                                                                                               
 \pinlabel {$1$} at 99 54                                                                                                               
 \pinlabel {$2$} at 86 100                                                                                                              
 \pinlabel {$3$} at 42 102                                                                                                              
 \pinlabel {$4$} at 24 57                                                                                                               
 
 \pinlabel {$gt(a_{0,1},a_{0,0},0,a_{0,1})$} at 280 -4   
 \pinlabel {$gt(a_{0,0},a_{0,3},1,a_{0,1})$} at 265 83                                                                                                              
 \pinlabel {$gt(a_{0,3},a_{0,1},2,a_{0,1})$} at 238 140                                                                                                             
 \pinlabel {$gt(a_{0,1},a_{0,4},3,a_{0,1})$} at 160 110                                                                                                              
 \pinlabel {$gt(a_{0,4},a_{0,1},4,a_{0,1})$} at 174 -5     
\endlabellist
\includegraphics[scale=1.00]{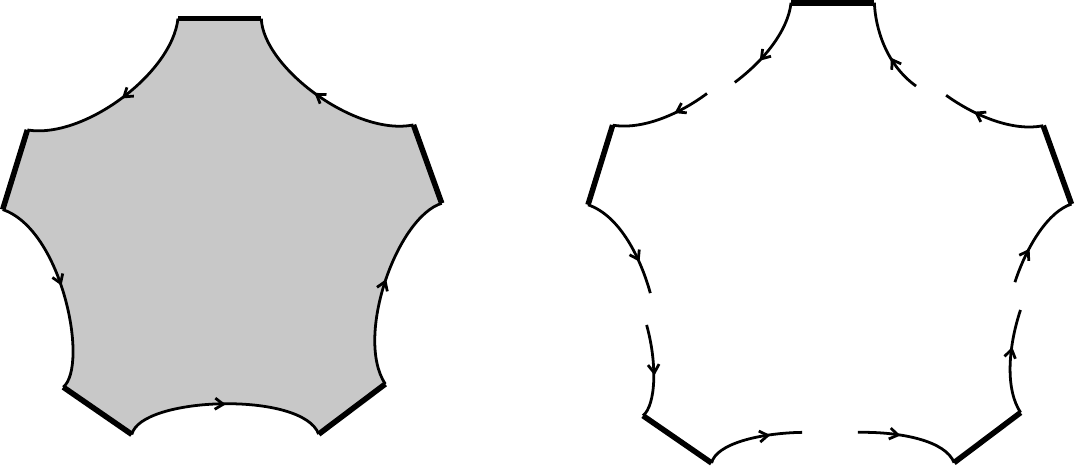}
\end{center}
\caption[Cutting a group polygon into group teeth]
{Cutting the group polygon from 
Figure~\ref{fig:scylla:group_poly} into group teeth.  
The bottom labeled side is arbitrarily chosen as the base.}
\label{fig:group_teeth}
\end{figure}

The group polygon 
$(x_0, e(x_0, x_1), x_1, \ldots, x_{o_j-1}, e(x_{o_j-1}, x_0))$ decomposes into 
the group teeth 
$gt(x_0, x_1, 0, x_0), gt(x_1, x_2, 1, x_0), \ldots, gt(x_{o_j-1}, x_0, o_j-1, x_0)$, 
where we have chosen $x_0$ as the base.
Group teeth are the pieces we get if we cut a group polygon 
in the middle of the labeled sides, and each piece records where it was and 
the label on the bottom (base) of the group polygon, but not the other labels.  
The base of the group polygon here is chosen arbitrarily; choosing a 
different labeled side to serve as the base gives a different 
decomposition of the group polygon into group teeth.  
Notice that there are at most $o_j|\Gamma|^3$ group teeth for the $j$th 
finite factor.

\subsection{The feasible polyhedron}

Using the same definitions from Section~\ref{sec:free_groups}, 
we define \emph{edges}, \emph{rectangles} and \emph{triangles} for $\Gamma$.  
Specifically, an edge is an ordered pair $e(x,y)$, where $x$ and $y$ are 
any letters in $\Gamma$.  A rectangle is an (unordered) pair $r(x,y)$, 
where $x,y$ are inverse letters in $\Gamma$ in the same infinite free factor.  
A triangle is $t(e_1, e_2, e_3)$, where the $e_i$ are edges satisfying the 
constraint that if $e_i = e(x,\Gamma_{i,j})$, then 
$e_{i+1} = e(\Gamma_{i,j-1}, y)$.
Triangles and rectangles are forbidden to have dummy edges, but group teeth 
can have them, for convenience.

We define a \emph{piece} to be a triangle, rectangle, or group tooth for $\Gamma$.
Let $\PP_\Gamma$ be the 
collection of all possible pieces.  We have $|\PP| \le (1 + \sum o_j)|\Gamma|^3 + |\Gamma|^2$.
Let $\EE_\Gamma$ be the collection of all non-dummy edges.  Let $V_\Gamma = \Q[\PP_\Gamma]$ and $E_\Gamma = \Q[\EE_\Gamma]$.  
There is a map $\partial: V_\Gamma \to E_\Gamma$ which is defined on generators by 
$\partial(r(x,y)) = e(x,y) + e(y,x)$, $\partial(t(e_1, e_2, e_3)) = e_1 + e_2 + e_3$, 
and $\partial( gt(x,y,n,z) ) = e(x,y)$.  
As before, we define 
$\iota: E_\Gamma \to E_\Gamma$ by 
$\iota(e(\Gamma_{i,j}, \Gamma_{k,l})) = -e(\Gamma_{k,l-1}, \Gamma_{i,j+1})$.
In the special case that 
the edge in a group tooth $gt$ is dummy, the boundary $\partial(gt)$ 
is defined to be $0$.

We need to make sure that the group teeth 
can be glued up into group polygons.  This requires more linear maps.  We define $\GG\EE_\Gamma$
to be the collection of tuples $(x,n,z)$, where $x,z$ are letters in $\Gamma$ 
in a finite factor $G_j$, and $n < o_j$.  
Let $GE_\Gamma = \Q[\GG\EE_\Gamma]$, and define $\partial_G :V_\Gamma \to GE_\Gamma$ 
on generators to be $0$ on triangles and rectangles, and set 
\[
\partial_G(gt(x,y,n,z)) = (y,(n+1)\bmod o_i,z) - (x,n,z).
\]
Given a collection $v\in V_\Gamma$ of group teeth, if $\partial_G(v)=0$, 
then the number of group teeth based at $z$ at position $n$ whose second 
labeled side is $x$ is the same as the number of group teeth based at $z$ 
at position $n+1$ whose first labeled side is $x$.  This restriction 
ensures that the group teeth can 
be glued up on their labeled sides to form group polygons.
The set of positive vectors in the subspace $\ker(\iota \circ \partial) \cap \ker \partial_G$ is a 
cone in $V_\Gamma$, which we denote by $C_\Gamma$.

To record the total degree of the boundary map for the surface maps 
that we will build, for each $w_i$, we define $N_i:V_\Gamma \to \Q$ 
which is $0$ on all triangles, $1$ on a rectangle $r(x,y)$ exactly when 
one of $x,y$ is the first letter of $w_i$, and $1$ on a group 
tooth $gt(x,y,n,z)$ exactly when $x$ is the first letter of $w_i$.  
The intersection of the cone $C_\Gamma$ with the affine subspace 
$\{v \in V_\Gamma \, |\, N_i(v) = 1\, \forall i\}$ is a polyhedron, 
which we call the \emph{admissible polyhedron} denote by $P_\Gamma$.

Finally, we want to compute Euler characteristic.  Define 
$\chi:V_\Gamma \to \Q$ on generators to be $0$ on rectangles, $-1/2$ on every 
triangle, and as follows on group teeth.    
For a group tooth $gt(x,y,n,z)$ in a finite factor $G_j$, we define
\[
\chi(gt(x,y,n,z)) = \left\{ \begin{array}{ll} 
         \frac{1}{o_j}       & \textnormal{ if $y$ follows $x$ cyclically in $\Gamma$} \\
         \frac{1}{o_j}-\frac{1}{2} & \textnormal{ otherwise }
         \end{array}\right.
\]
\begin{theorem}\label{thm:main_scylla}
In the above notation, $\scl(\Gamma) = \inf_{P_\Gamma}-\chi(v)/2$.  
Furthermore, an extremal surface for $\Gamma$ can be extracted from a 
minimizing vector in $P_\Gamma$.  The vector space $V_\Gamma$ has dimension 
at most $|\Gamma|^3(1+\sum_j o_j) + |\Gamma|^2$, and there are 
at most $|\Gamma|^2(1+\sum_j o_j)$ equality constraints which cut out $P_\Gamma$.
\end{theorem}

The proof of Theorem~\ref{thm:main_scylla} breaks into two 
inequalities.  We do the easy direction first.

\begin{lemma}\label{lem:scylla_part_1}
Given $v \in P_\Gamma$, there is a surface map $f:S \to X_G$ 
admissible for $\Gamma$ with $-\chi^-(S)/2n(S,f) \le -\chi(v)/2$.
\end{lemma}
\begin{proof}
There is some $k \in \Z$ so that $kv$ is integral, and therefore 
represents a collection of pieces.  First, because 
$\partial_G(kv) = 0$, we can glue up the group teeth into 
group polygons, as follows.  Consider $z$ a letter in 
$\Gamma$ in finite factor $G_j$, and consider all the 
group teeth based at $z$ (of the form $gt(*,*,*,z)$).  For a given $0\le n<o_j-1$, 
and for any letter $s$, there are as many group teeth 
of the form $gt(*,s,n,z)$ as there are of the form $gt(s,*,n+1,z)$  
Therefore, we can glue them arbitrarily.  To glue 
two group teeth together means to undo the cutting shown 
in Figure~\ref{fig:group_teeth}, i.e. it means to 
identify the second labeled side of the first with the first 
labeled side of the second to obtain a string of $3$ labeled 
sides and two edges (alternating).  The result of the arbitrary gluing 
is a collection of strings of $o_j$ glued group teeth, with 
each end unglued. Now consider $n=o_j-1$.  
Recall that we required that a group tooth at index $0$  
be of the form $gt(z,*,0,z)$, i.e. begin with letter $z$, 
and we required that a group tooth at index $o_j-1$ be of the 
form $gt(*,z,o_i,z)$, i.e. end with letter $z$.  Therefore, 
every string of $o_j$ group teeth can be glued up into a loop 
of length $o_j$.  That is, every string can be glued up 
into a group polygon.  Doing this for every $z$ 
collects all the group teeth and glues them all up into group 
polygons, so we are left only with rectangles, triangles, and 
group polygons.  

Because $\iota\circ \partial(kv) = 0$, we can glue the rectangles, 
triangles, and group polygons along edges into a surface $S$.  
We ignore dummy edges, and they remain unglued, which is the 
correct behavior. 
There is a canonical map $f:S \to X_G$ which sends every triangle 
to the basepoint, every rectangle around the free factor loops, and 
every group polygon for factor $G_j$ to the appropriate 
$2$-cell.  By construction, $S$ is admissible for $\Gamma$, and 
$n(S,f) = k$.  

It is possible that there are branch points in the map $f$ which are 
produced by the triangle gluing.  However, there is a spine $T$ in 
$S$, with one vertex for every rectangle, 
triangle, and group polygon, and 
one edge for every glued edge.  Because $f$ can only 
introduce branch points, which only increases Euler characteristic, 
we have $-\chi^-(S) \le -\chi^-(T)$.  Now we show that $\chi(T) = \chi(kv)$.
Consider one of the group polygons $p$ in $S$ associated with 
finite factor $G_j$.  There will be 
some dummy edges on $p$, and some real edges glued to other 
pieces.  Let $m$ be the number of real edges.  
The contribution to Euler characteristic from $p$ is $1-m/2$.  
But by construction, the value of the linear function $\chi$ 
on the sum of the group teeth in $p$ will be 
$m(1/o_j - 1/2) + (o_j-m)(1/o_j) = 1-m/2$.  Similarly, 
the linear function $\chi$ correctly computes the contribution to 
Euler characteristic from the rectangles and triangles.
Therefore, $\chi(T) = \chi(kv)$.  Here we must be careful, since 
a priori, it is possible that $\chi^-(T) \ne \chi(T)$, 
because there might be disk components.  Recall, though, that we are assuming 
that $\Gamma$ is reduced and has no finite abelian loops, so in fact 
$\chi^-(T) = \chi(T) = \chi(kv)$, and we have
\[
\frac{-\chi^-(S)}{2n(S,f)} \le \frac{-\chi(kv)}{2k} = -\frac{1}{2}\chi(v)
\]
\end{proof}

\begin{lemma}\label{lem:scylla_part_2}
Given a surface map $f:S \to X_G$ which is admissible for $\Gamma$, 
there is a vector $v \in P_\Gamma$ with $-\chi(v)/2 \le -\chi^-(S)/2n(S,f)$.
\end{lemma}
\begin{remark}
This proof follows the initial strategy of~\cite{Calegari_sails}, 
but takes a different tack halfway through.
\end{remark}
\begin{proof}

Recall $G = *_j G_j$, where $G_j$ is finite or infinite 
cyclic; let $g_j$ be the generator of $G_j$, so $g_j$ has order $o_j$.  
Let $K_j$ be the wedge summand in $X_G$ corresponding 
to $G_j$, so $K_j$ is either a copy of $S^1$, if $G_j$ is infinite,  
or a copy of $S^1$ with a disk glued to it by an $o_j$-fold covering 
map on the boundary.  Denote the basepoint of $X_G$ by $*$.
There is a canonical choice of a 
map $\alpha_j:[0,1] \to X_G$ representing $g_j$ in $G$, which is 
simply the map positively traversing the copy of $S^1$ in $K_j$.

Recall we have written $\Gamma = \sum_i w_i$ in reduced form 
and so that only positive powers of the finite factor generators appear.
Formally, $\Gamma$ is a map $\Gamma : \coprod_i S_i^1 \to X_G$.  Separate 
the components into maps $\gamma_i:S_i^1 \to X_G$, where $\gamma_i$ 
represents $w_i$.  After homotopy, we take it that each $\gamma_i$ 
is a concatenation of the paths $\alpha_j$.  This induces 
an oriented simplicial structure on $S_i^1$, where the vertices 
are the preimage of $*$, and on each $1$-cell, $\gamma_i$ 
performs the map $\alpha_j$ (or possibly $\alpha_j^{-1}$, in the 
case of an infinite factor) for some $j$.  

After homotopy, we may assume that the map 
$\partial f: \partial S \to X_G$ factors through our chosen 
representative map $\Gamma = \coprod_i \gamma_i$, so 
there is a simplicial structure on $\partial S$.  
Now consider $f^{-1}(*)$ in $S$.  After homotopy (retaining the fact 
that $\partial f$ factors through $\Gamma$), we may assume that 
$f^{-1}(*)$ is a collection of essential loops and 
arcs in $S$.  We remove the loops by compressing $S$, 
so we assume that $f^{-1}(*)$ is a collection of essential 
arcs in $S$.  These arcs are compatible with the 
simplicial decomposition of $\partial S$, in the sense that every arc 
begins and ends on a vertex in $\partial S$.

We write $S = \cup_j S_j$, 
where $S_j$ is the collection of (closures of) components of 
$S \setminus f^{-1}(*)$ which map into $K_j$.  We write 
$S_j = \cup_k T_{j,k}$, where $T_{j,k}$ is the closure of a single 
component of $S \setminus f^{-1}(*)$.  Each $T_{j,k}$ 
comes with a map to $K_j$, and each $T_{j,k}$ has a simplicial 
structure on its boundary, where the vertices all map to 
$*$, and the $1$-cells are in $\partial S$ (mapping to 
$\alpha_j$), or
arcs in $f^{-1}(*)$.  We call the $1$-cells on $T_{j,k}$ which 
are arcs on $\partial S$ \emph{sides}, and 
we call $1$-cells which are arcs in the preimage $f^{-1}(*)$ \emph{edges}.  
Sides are part of $\partial S$, and edges get glued 
to other edges to produce $S$.
The $1$-cells need not 
alternate between edges and sides; for example, 
part of $\partial T_{j,k}$ might be two copies of $\alpha_j$ 
in a row, which yields two sides in a row.  
Whenever this occurs, blow up the vertex to produce 
a \emph{dummy edge} between these two sides.  Our 
usage of side, edge, and dummy edge coincides with the 
usage in the definitions of rectangles, triangles, and group teeth.
 
\begin{figure}[ht]
\begin{center}
\labellist
\small\hair2pt
 \pinlabel {$e_1$} at 241 239
 \pinlabel {$e_2$} at 170 184
 \pinlabel {$a_2$} at 225 199
 \pinlabel {$a_1$} at 185 228
\endlabellist
\includegraphics[scale=0.6]{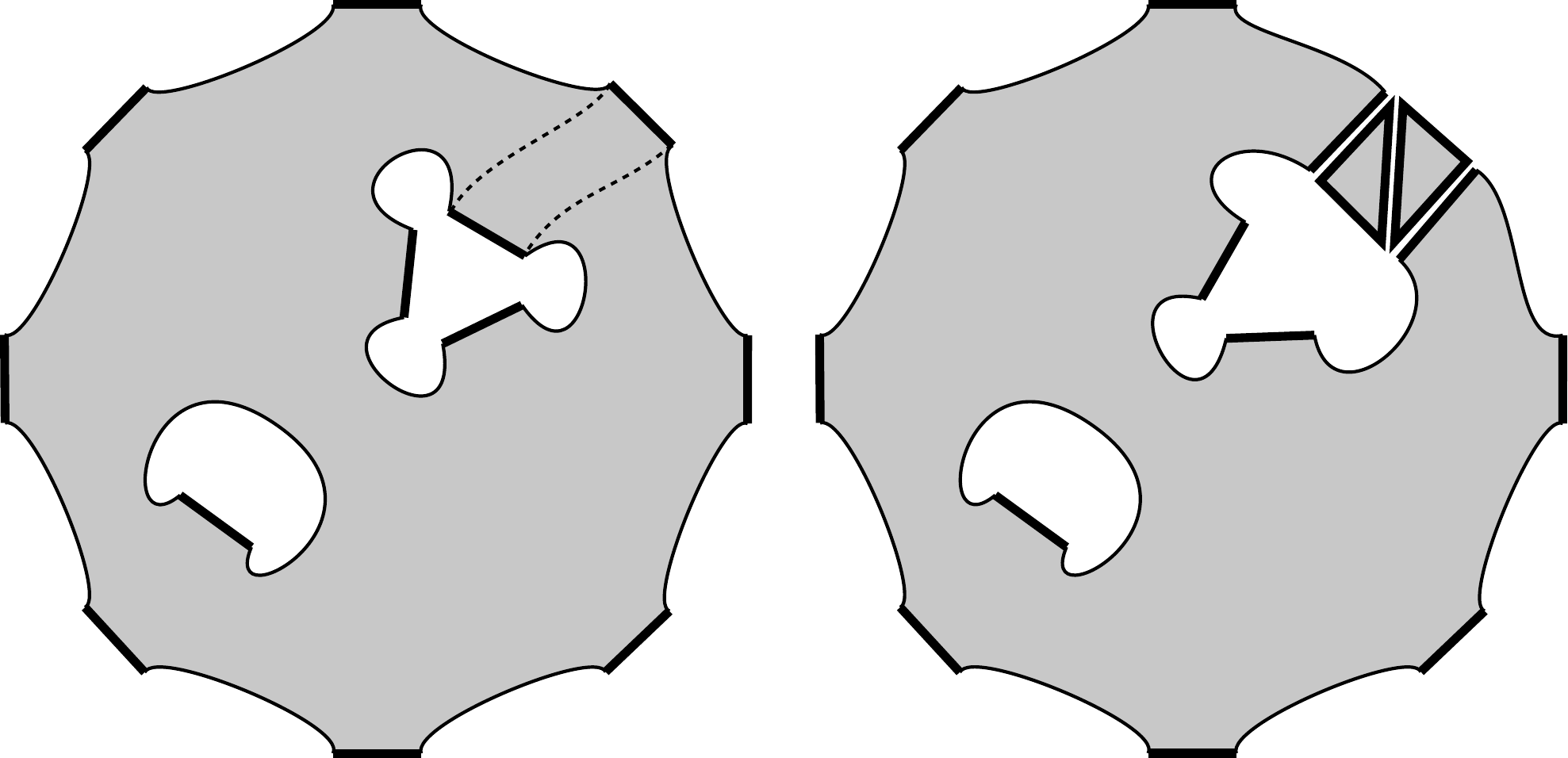}
\caption{Cutting the polygonal surface $T_{j,k}$ to reduce the 
number of boundary components.  Sides are thin, and edges are bold.}
\label{fig:cut_to_disk}
\end{center}
\end{figure}
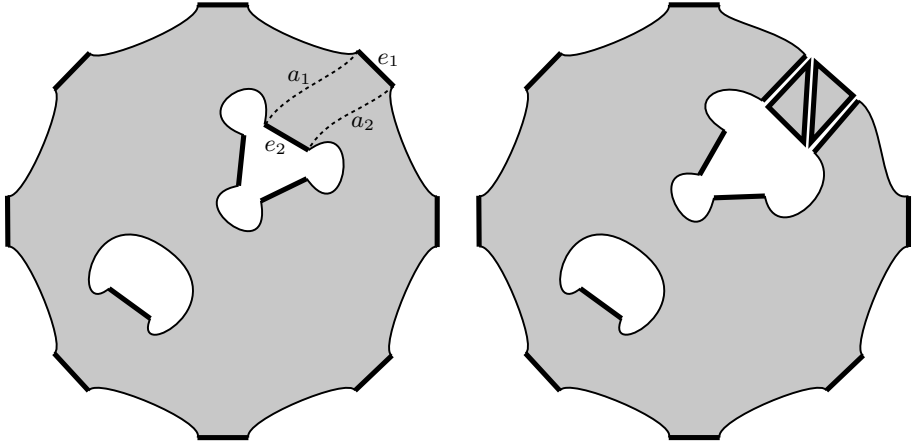

We have decomposed $S$ into a union of components $T_{j,k}$, where 
each $T_{j,k}$ has boundary alternating between edges and sides.  
Now we will focus on each $T_{j,k}$ and simplify it.  
Denote a particular $T_{j,k}$ by $T$.  
If $T$ is not a planar surface, then since $G_j$ is 
abelian, we can compress it, so we may assume that 
$T$ is planar.  Our goal is to write $T$ as 
a union of triangles and (as applicable) rectangles or 
group polygons.  The first step is to rearrange $T$ 
to have a single simplicial boundary component.  
Suppose $T$ has more than one boundary component.  
Then pick edges $e_1$ and $e_2$ on different
boundary components.  Because $T$ is connected, 
there is a parallel pair of arcs $a_1$, and $a_2$ connecting the terminal 
vertex of $e_1$ to the initial vertex of $e_2$, and vice versa, such 
that $a_1$ and $a_2$ bound a strip in $T$.  
Cut $T$ along these arcs.  The result is a square with oriented 
boundary $(e_1, a_1, e_2, a_2)$, plus the remaining surface of $T$, 
which we again denote by $T$.  
Since we cut out a strip, the surface $T$ now has one fewer 
boundary component,
and $\partial T$ still has a simplicial structure alternating between 
edges and sides, 
where the cut arcs $a_1$ and $a_2$ have become edges.  We must be careful, though:
a priori, the arcs $a_1$ and $a_2$ need not be nullhomotopic, 
so we cannot trivially replace them on the boundary by edges, which by definition 
map to the basepoint.  The key is that $G_j$ is abelian, so this homotopy 
actually \emph{is} possible.
Similarly, we may assume that all four sides of 
the square are edges, so we can cut this square into 
two triangles.  In the special case that one of the edges $e_1$ or $e_2$ 
is a dummy edge, then the square simplifies to a triangle, and 
if $e_1$ and $e_2$ are both dummy edges, the square can be removed 
completely leaving just a normal gluing edge that we have cut.  
Again, the remaining surface $T$ has one fewer 
boundary component, and it has the same simplicial structure
on its boundary alternating between edges and sides.  Furthermore, 
the remaining $T$ plus the square presents the same set of edges; in other words, 
we haven't changed how the piece $T$ interacts with the other pieces $T_{j,k}$.
After finitely many cuts, then, we may assume that $T$ has a single boundary 
component.  See Figure~\ref{fig:cut_to_disk}.

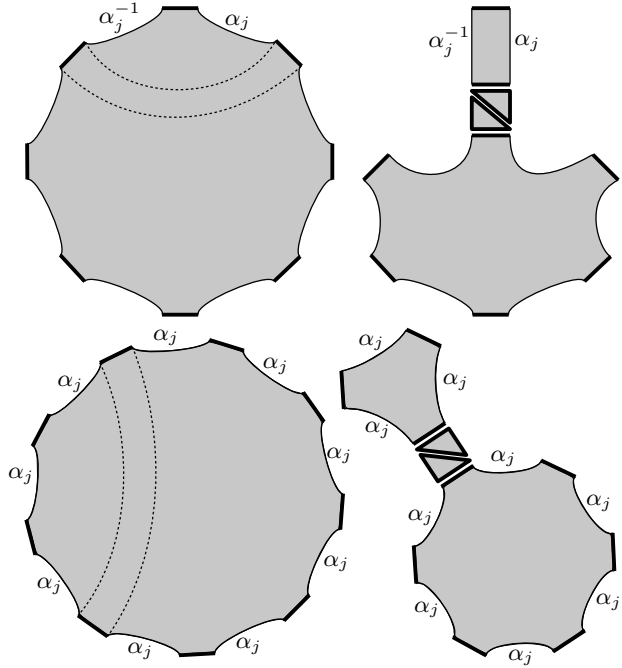
\begin{figure}[ht]
\begin{center}
\labellist
\small\hair 2pt
 \pinlabel {$\alpha_j$} at 134 399
 \pinlabel {$\alpha_j^{-1}$} at 60 402
 \pinlabel {$\alpha_j$} at 316 386
 \pinlabel {$\alpha_j^{-1}$} at 267 387
 
 \pinlabel {$\alpha_j$} at 29 172
 \pinlabel {$\alpha_j$} at -2 114
 \pinlabel {$\alpha_j$} at 16 45
 \pinlabel {$\alpha_j$} at 73 4
 \pinlabel {$\alpha_j$} at 141 7
 \pinlabel {$\alpha_j$} at 196 59
 \pinlabel {$\alpha_j$} at 198 127
 \pinlabel {$\alpha_j$} at 158 183
 \pinlabel {$\alpha_j$} at 89 201

 \pinlabel {$\alpha_j$} at 218 198
 \pinlabel {$\alpha_j$} at 222 143
 \pinlabel {$\alpha_j$} at 271 173
 
 \pinlabel {$\alpha_j$} at 249 88
 \pinlabel {$\alpha_j$} at 253 30
 \pinlabel {$\alpha_j$} at 312 0
 \pinlabel {$\alpha_j$} at 367 35
 \pinlabel {$\alpha_j$} at 361 97
 \pinlabel {$\alpha_j$} at 301 125
\endlabellist
\includegraphics[scale=0.6]{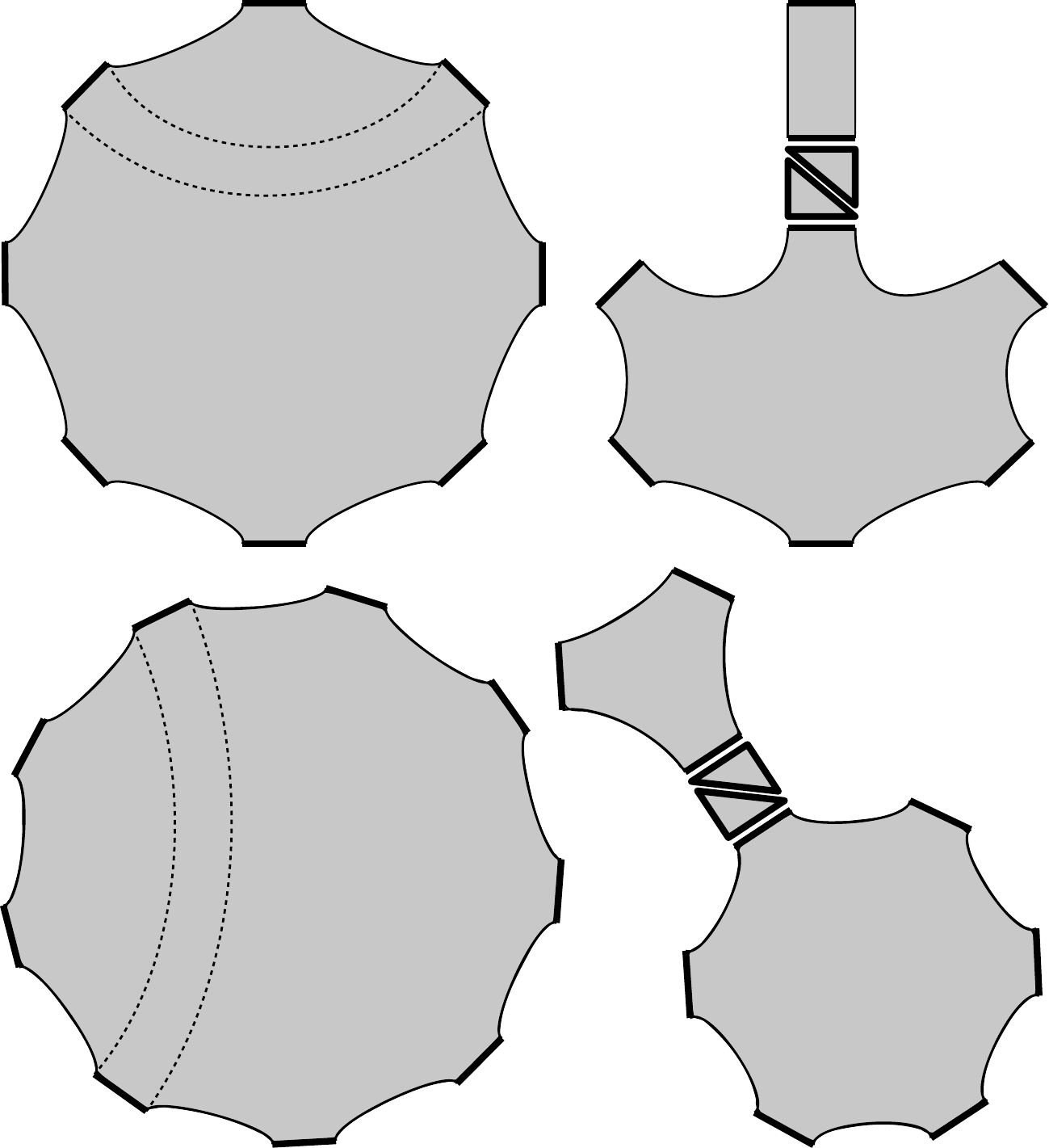}
\caption{Reducing the size of the boundary of $T_{j,k}$ 
by pinching off a rectangle in the case of an infinite factor, top, 
or a group polygon in the case of a finite factor, bottom.  The 
finite factor here is $G_j = \Z/3\Z$.}
\label{fig:pinch}
\end{center}
\end{figure}

In the case that $G_j$ is finite, 
every side of $T$ must be labeled with the same, positive, 
arc $\alpha_j$.  Furthermore, in order for 
$T$ to map into $K_j$, we note that the number of sides 
of $T$ must be a multiple of $o_j$.  Number the edges of $T$ 
$e_0, \ldots, e_n$.  Perform the same square-cutting 
move on the edges $e_0$ and $e_{o_j}$.  This cuts $T$ into 
two disk components, one of which has exactly $o_j$ sides, and 
one of which has $n-o_j$ sides.  Note that the former component 
is exactly a group polygon, and both components still have 
boundary alternating between edges and sides, so we can repeat 
this procedure on the larger component recursively.  
The result is that we have decomposed 
$T$ into a union of triangles and group polygons.  

If $G_j$ is infinite cyclic, then each side of $T$ is a copy of 
$\alpha_j$ or $\alpha_j^{-1}$.  There must be some edge 
which lies between sides labeled $\alpha_j$ and $\alpha_j^{-1}$.  
Pinch these two sides together to form a rectangle, 
and cut it off $T$.  Where we cut, it is possible that 
$T$ has several adjacent edges.  Using triangles, we can reduce these 
to a single edge.  This produces a new $T$ with two fewer 
sides.  Repeating this procedure decomposes $T$ into a union of 
rectangles and triangles.  Another way to see this is to reduce $\partial T$, 
thought of as a cyclic word.  The boundary $\partial T$ must be trivial, 
so it pinches together to form a tree, which has an obvious structure as a 
union of triangles and rectangles.  See Figure~\ref{fig:pinch}.

For both decompositions, we never changed the topological type of $S$; 
we simply homotoped $f$ to give the pieces of $S$ a combinatorial structure.  
There is a combinatorial spine $L$ to which $S$ deformation retracts, given by 
a vertex for every piece (rectangle, triangle, or group polygon), and 
an edge for every gluing edge.  Therefore, $\chi(S) = \chi(L)$.  Again 
we use the assumption that there are no finite abelian loops in $\Gamma$, so 
$\chi^-(S) = \chi(S) = \chi(L)$.  Now 
cut every group polygon into group teeth, and let $v \in V_\Gamma$ be 
the vector which records how many of each type if piece we have.  Clearly, 
$v \in C_\Gamma$, since we obtained $v$ by cutting a surface, and 
$N_i(v) = n(S,f)$ for all $i$, and $\chi(v) = \chi(L)$.  
Then $(1/N_i(v))v \in P_\Gamma$, and 
\[
-\frac{1}{2}\chi\left(\frac{1}{N_i(v)}v\right) 
 = -\frac{1}{2n(S,f)}\chi(L) \le \frac{-\chi^-(S)}{2n(S,f)}
\]
\end{proof}
\begin{remark}\label{rem:proof_methods}
There are (at least) two different ways to prove Lemma~\ref{lem:scylla_part_2}. 
One is to follow our strategy of cutting at the preimage of the basepoint and 
rearranging the resulting surface pieces.  This method follows 
the strategy of \cite{Calegari_sails}.  The other is to 
remove the preimage of small neighborhoods of central points in the $2$ 
cells in $X_G$.  The resulting surface maps into a free group, and it is 
straightforward to see that it has boundary $\Gamma$, plus some
relators.  We get a fatgraph representative, and then 
glue back in disks for each relator while preserving a combinatorial 
structure.  This latter method is not as simple to state rigorously, 
which is why we used the former for our proof, but it does 
generalize more readily to groups with more complicated presentations.
See Section~\ref{sec:extensions}.
\end{remark}

\begin{remark}
Lemma~\ref{lem:scylla_part_2} gives an independent proof 
that surface maps into free groups factor through labeled 
fatgraph maps (possibly after compression); i.e. 
Lemma~\ref{lem:fatgraphs_rep_surface_maps}.
\end{remark}

\begin{proof}[Proof of Theorem~\ref{thm:main_scylla}]
Theorem~\ref{thm:main_scylla} is immediate from Lemmas~\ref{lem:scylla_part_1} 
and~\ref{lem:scylla_part_2}, plus the same observation from the 
proof of Proposition~\ref{prop:free_groups}; namely, that 
gluing a minimizing vector for $-\chi/2$ in $P_\Gamma$ 
produces an extremal surface.  Similarly, this 
surface will have no branch points.  If $\Gamma$ 
has abelian loops, then we must add some disk components to 
this extremal surface.  To make the combinatorial 
structure complete, we can tile these disks with group polygons.

The size of the linear programming problem is obviously linear 
in the input size, the number of factors, and the orders of the finite factors, 
because there are only polynomially-many types of pieces. 
However,  we rewrote the input word 
using only positive powers of the generators, so we increased the 
input size before running the algorithm.  This is still polynomial; 
however, in practice, and to achieve the complexity that we asserted, 
we reduce the input words completely (allowing inverses of the generators 
of the finite factors), and we build surfaces out of a wider class of pieces.  
Namely, we must allow group teeth for positive and negative finite factor letters, 
and we allow rectangles in the finite factors.  The only modification 
necessary in the proofs of Lemmas~\ref{lem:scylla_part_1} 
and~\ref{lem:scylla_part_2} is that in the reduction of the finite-factor 
disks $T_{j,k}$, we must first pinch off rectangles as in the infinite-factor disks, 
and then pinch off group polygons.
This keeps the computational 
complexity linear in the orders of the finite factors.
\end{proof}

\section{Experimental and theoretical applications}
\label{sec:extensions}

\subsection{Examples}

The \texttt{scylla} surface decomposition gives a 
combinatorial structure to surface maps.  
Figure~\ref{fig:scylla_surfaces} shows two examples of surface maps 
into $\Z/3\Z * \Z/2\Z \cong \PSL(2,\Z)$.  

\begin{figure}[ht]
\begin{center}
\labellist
\small\hair2pt
 \pinlabel {$a_{0,0}$} at -6 34
 \pinlabel {$b_{0,1}$} at 147 -10
 \pinlabel {$a_{0,0}$} at 308 45
 \pinlabel {$b_{0,1}$} at 149 162
 \pinlabel {$a_{0,0}$} at 77 46
 \pinlabel {$b_{0,1}$} at 146 34
 \pinlabel {$a_{0,0}$} at 220 45
 \pinlabel {$b_{0,1}$} at 148 54
 \pinlabel {$a_{0,0}$} at 77 111
 \pinlabel {$b_{0,1}$} at 142 98
 \pinlabel {$a_{0,0}$} at 222 109
 \pinlabel {$b_{0,1}$} at 147 117
 
 \pinlabel {$a_{0,0}$} at 408 18
 \pinlabel {$a_{0,1}$} at 482 18
 \pinlabel {$a_{0,3}$} at 444 65
 \pinlabel {$a_{0,0}$} at 408 107
 \pinlabel {$a_{0,1}$} at 482 106
 \pinlabel {$a_{0,3}$} at 439 154
 \pinlabel {$b_{0,4}$} at 323 96
 \pinlabel {$b_{0,4}$} at 370 92
 \pinlabel {$b_{0,2}$} at 518 87
 \pinlabel {$b_{0,2}$} at 563 110
\endlabellist
\includegraphics[scale=0.55]{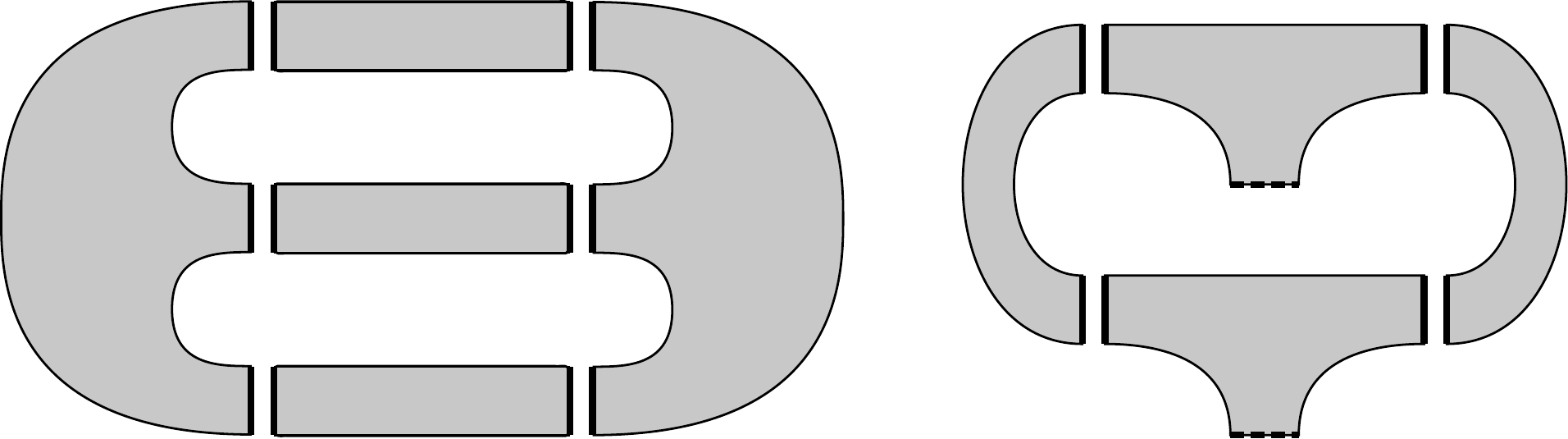}
\end{center}
\caption[Examples of \texttt{scylla} output surfaces.]
{Surfaces bounding the chains $ab$ 
and $aabab$ in the group $G = \Z/3\Z *\Z/2\Z$.  
These exhibit $\scl_G(ab) \le 1/12$ and $\scl_G(aabab) = 0$.  
In fact, they are both extremal.  The dashed edges are dummy edges.  
Note the group polygons in $\Z/2\Z$ look like rectangles, 
but they aren't.}
\label{fig:scylla_surfaces}
\end{figure}
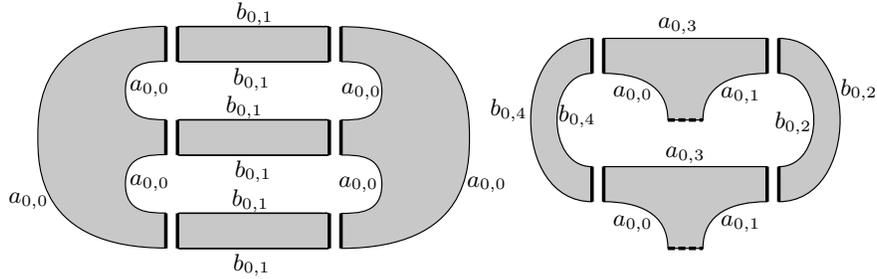

\subsection{Relationship to \texorpdfstring{$\scl$}{scl} in free groups}

The groups $\Z/n\Z *\Z/m\Z$ approximate $\Z*\Z$ locally 
as $n$ and $m$ get large, since balls of a fixed radius in the Cayley 
graph are eventually the same.  However, the groups are certainly 
different, and if $\Gamma \in B_1^H(\Z*\Z)$, then it's not 
immediately clear if there is a relationship between 
$\scl_{\Z*\Z}(\Gamma)$ and $\scl_{\Z/n\Z*\Z/m\Z}(\Gamma)$.  
In fact, if $\Gamma$ is fixed, then taking one of the finite-factor 
generators to have very large order causes $\scl$ to behave 
as though the generator has infinite order.  There are 
different ways to state this fact; for simplicity, we 
state the version which takes all the orders to infinity.

\begin{proposition}\label{prop:finite_approx}
Let $H_j = \langle a_j \rangle = \Z$ and 
$G_j = \langle a_j \rangle / \langle a_j^{o_j} \rangle = \Z/o_j\Z$.   
Let $H = *_j H_j$ and $G = *_j G_j$.
If $\Gamma \in B_1^H(H)$, let $|\Gamma|_j$ denote the 
number of letters in $\Gamma$ from $H_j$.  We have 
\[
\scl_G(\Gamma) \le \scl_H(\Gamma) \le \scl_G(\Gamma) + \sum_j\frac{|\Gamma|_j}{2o_j}
\]
\end{proposition}
\begin{proof}
The inequality on the left is immediate, since $G$ is a quotient of $H$, so it 
remains to bound $\scl_H(\Gamma)$ in terms of $\scl_G(\Gamma)$.  Let 
$S$ be an extremal surface for $\Gamma$ in $G$ mapping with degree $N$, 
so $\partial S = N\Gamma$.  By the 
proof of Theorem~\ref{thm:main_scylla}, we can put $S$ into 
the combinatorial \texttt{scylla} form.  Recall in this form, 
we cyclically reduce $\Gamma$ completely, and we allow group polygons 
labeled with positive and negative powers of generators.  If $\Gamma$ 
has abelian loops, $S$ will have disk components decomposed into 
group polygons.  Let there be $K_j$ total group polygons in $S$ 
in factor $G_j$.  Remove a 
small neighborhood of a central point in every group polygon in $S$ to 
obtain a surface $S'$.  The surface $S'$ decomposes into 
triangles and rectangles (each group polygon in $G_j$ being replaced by 
$o_j$ rectangles), so $S'$ comes with a map to $H$.  
The boundary of $S'$ decomposes into 
$N\Gamma + \sum_j m_ja_j$ in $B_1^H(H)$.  However, since both 
$N\Gamma$ and $\partial S'$ are homologically trivial in $H$, 
we must have $m_j=0$ for all $j$.  In other words, there must be exactly as 
many positive group polygons in $G_j$ as there are negative 
group polygons in $G_j$.  Therefore, we can pair them up and glue on annuli 
to obtain a surface $S''$ with a map to $H$ such that $\partial S'' = N\Gamma$.  
Gluing the annuli does not affect Euler characteristic, but removing the neighborhoods 
of the central points in the group polygons does.  We have 
\[
-\chi^-(S'') \le -\chi^-(S) + \sum_jK_j.
\]
This would be an equality, except removing a central point from 
a group polygon which is a disk component of the surface does not affect 
$\chi^-$.
Every group polygon in factor $G_j$ contains $o_j$ letters from factor $G_j$ 
in $\Gamma$.  Therefore, $K_j \le N|\Gamma|_j/o_j$, so
\[
\scl_H(\Gamma) \le \frac{-\chi^-(S'')}{2N} 
               \le \frac{-\chi^-(S)}{2N} + \sum_j\frac{N|\Gamma|_j}{2No_j} 
               = \scl_G(\Gamma) + \sum_j\frac{|\Gamma|_j}{2o_j} 
\]
\end{proof}
\begin{remark}
In any small cancellation group, we can control the number of relators in 
a surface in terms of $N|\Gamma|$ and the small cancellation constant.  
For free products of cyclic groups, we can glue annuli to remove the 
newly created boundary 
components in $S'$.  In a general small cancellation group, there is no guarantee 
that there are the same number of a relator and its inverse, so filling in the 
boundary could add Euler characteristic proportional to the small 
cancellation constant, so this argument will not work.
If we require that the relators be linearly independent in homology, 
though, this will force the relators to match up, and the analogue of 
Proposition~\ref{prop:finite_approx} should be true.
\end{remark}

As a corollary, we can compute $\scl([a,b])$ when only $a$ has 
finite order.
\begin{corollary}\label{cor:abAB_formula}
Let $G = \langle a \rangle / \langle a^o \rangle * \langle b \rangle$.  
Then $\scl_G([a,b]) = 1/2 - 1/o$.
\end{corollary}
\begin{proof}
Build an admissible surface for $[a,b]$ using two group polygons, one 
with $a$ and one with $A$.  Cyclically order the edges on these polygons in 
opposite directions, and connect matching edges with group rectangles for $b$.  
The resulting surface $S$ is admissible for 
$[a,b]$, and has boundary which is $o$ copies of $[a,b]$, so $n(S) = o$.  
Directly computing $\chi(S)$, we get
\[
\frac{-\chi^-(S)}{2o} = \frac{o-2}{2o} = \frac{1}{2} - \frac{1}{o}
\]
Since $\scl_{\Z*\Z}([a,b]) = 1/2$, the right hand inequality from 
Proposition~\ref{prop:finite_approx} is an equality, and 
this surface is extremal.
\end{proof}

\subsection{Experiments}

Stable commutator length has connections to number theory, and it
often behaves in a \emph{quasipolynomial} way.  A function 
$f:\Z \to \R$ is a quasipolynomial if there exists $D \in \N$, 
called the \emph{period} of $f$, and polynomials $p_0, \ldots, p_{D-1}$ 
so that $f(n) = p_{n\%D}(n)$.  Here $n\%D$ denotes the integer remainder.  
Another way to think of a quasipolynomial is as a polynomial whose 
coefficients depend on the residue class of the input.  If a quasipolynomial 
has degree $1$ (for all residue classes), we call it a 
\emph{quasilinear} function.  See~\cite{Calegari_sails} 
and~\cite{CW_quasipoly}.

\begin{conjecture}\label{conj:main_conj}
Let $G_j^k = \langle a_j \rangle / \langle a_j^{o_{j,k}} \rangle = \Z/o_{j,k}\Z$ 
and $G_k = *_j G_j^k$.  Let $\Gamma$ be fixed so that $\Gamma \in B_1^H(G_k)$ for all $k$ 
(this only requires that $\Gamma$ have the appropriate number of inverse pairs in 
the free factors).  Then for $o_{j,k}$ sufficiently large, $\scl_{G_k}(\Gamma)$ 
is piecewise-quasilinear in the variables $1/o_{j,k}$.
\end{conjecture}
Conjecture~\ref{conj:main_conj} summarizes the results of experiments below
and complements the main theorem in~\cite{CW_quasipoly}.  
We now give specific examples illustrating quasilinear behavior.  
In all cases, we take $G$ to be generated by $a,b,\ldots$ with 
orders $o_1, o_2, \ldots$.  Brackets indicate coefficients depending on the 
residue class of the input, so for example the $\{2/3,1/2\}$ appearing 
in the formula for $\scl_G(abAABB+ab)$ means that the coefficient of $1/\min(o_1, o_2)$ 
is $2/3$ if $\min(o_1, o_2)\%2 = 0$ and $1/2$ if $\min(o_1, o_2)\%2 = 1$.
All of these formulas have been experimentally verified for thousands of values.

\vspace{5mm}
\begin{tabular}{rll}
$\scl_G([a,b])$ & $= 1/2 - 1/\min(o_1, o_2)$ & if $\min(o_1,o_2) \ge 2$ \\ [1.5ex]
$\scl_G(abAABB + ab)$ & $= 2/3 - \{2/3,1/2\}/\min(o_1, o_2)$ & if $\min(o_1,o_2) \ge 2$ \\ [1.5ex]
$\scl_G(abAAABBB)$ & $= 3/4 - 1/o_1 - 1/o_2$ & if $\min(o_1, o_2) \ge 7$ \\ [1.5ex]
$\scl_G(aabABAAbaB)$ & $= 1/2 - \{2,1\}/o_1$ & if $\min(o_1, o_2) \ge 3$ \\ [1.5ex]
$\scl_G(aba^2b^2a^3b^3A^5B^5)$ & $=1 - \frac{1}{2o_1} - \frac{1}{2o_2}$ & if $\min(o_1, o_2) \ge 6$ \\ [1.5ex]
\end{tabular}
\vspace{5mm}

Small changes can strongly affect $\scl$: above we 
see $\scl_G(aba^2b^2a^3b^3A^5B^5)$ has a simple formula.  However, 
$\scl_G(aba^2b^2a^3b^3A^6B^6)$ exhibits behavior which is 
piecewise periodic in both orders with period $6$.  For simplicity, 
we plot the value of $\scl$ for fixed $o_1=100$ as $o_2$ varies.  
See Figure~\ref{fig:periodic_scl}.

\begin{figure}[ht]
\labellist
\small\hair 2pt
 \pinlabel {$10$} at 50 -3
 \pinlabel {$20$} at 152 -3
 \pinlabel {$30$} at 254 -3
 \pinlabel {$40$} at 357 -3
 \pinlabel {$50$} at 458 -3
 \pinlabel {$1.01$} at -14 91
 \pinlabel {$1.02$} at -14 174
 \pinlabel {$1.03$} at -14 257
 \pinlabel {$o_2$} at 490 0
 \pinlabel {$\scl$} at -6 290
\endlabellist
\includegraphics[scale=0.55]{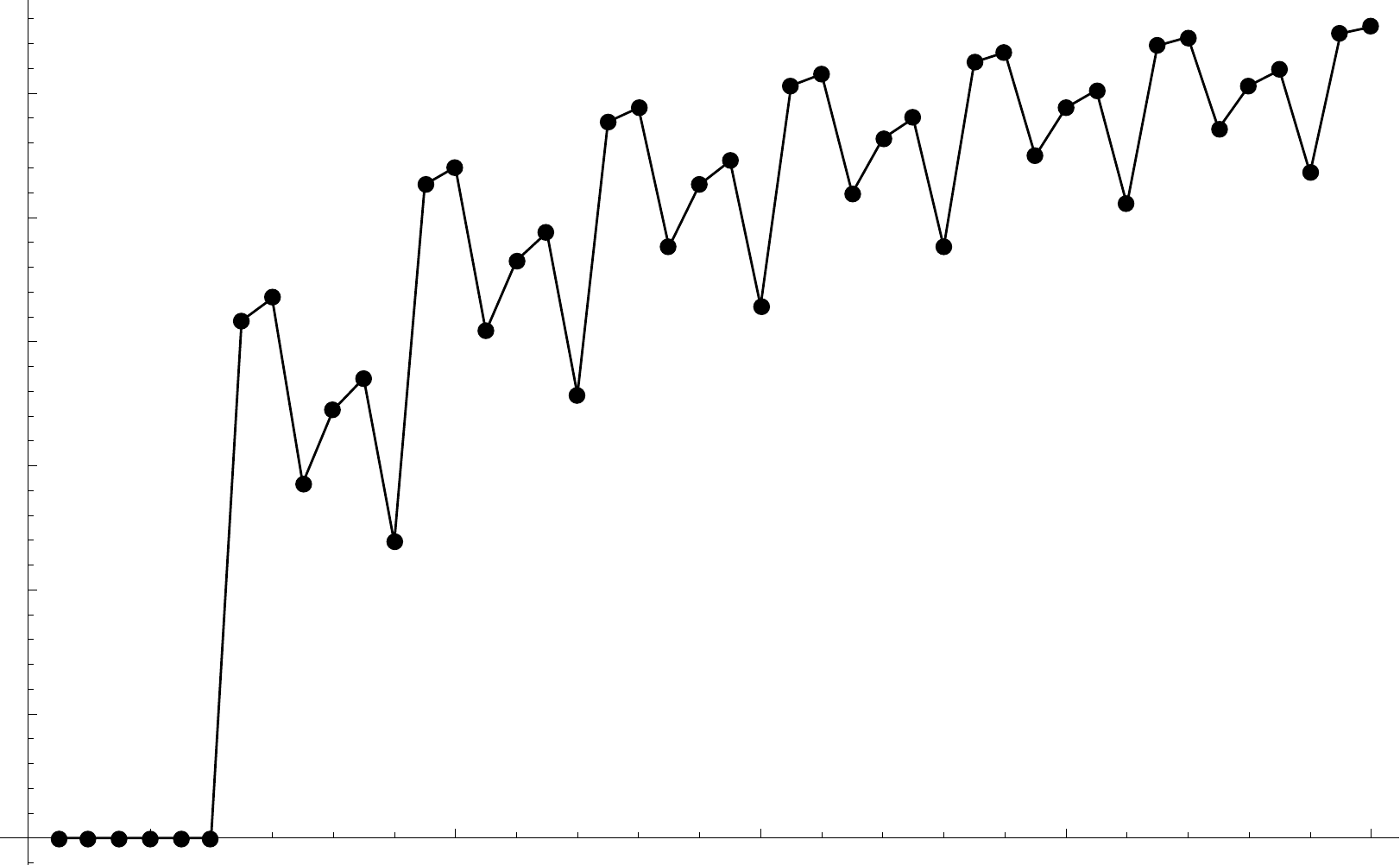}
\caption{The value of $\scl_G(aba^2b^2a^3b^3A^6B^6)$ as $o_2$ varies, 
for fixed $o_1=100$.  Note the initially flat behavior, followed by 
quasi-linearity in $1/o_2$ with period $6$.}
\label{fig:periodic_scl}
\end{figure}
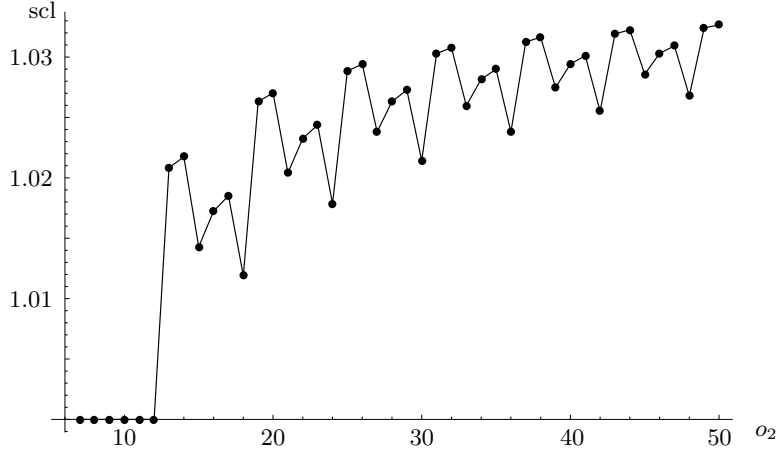

For $\scl_G(abABacAC)$, the formula has four pieces, each of which is quasilinear.  
It is clearest to fix $o_1$ and draw the regions in the $o_2$-$o_3$ plane 
with the appropriate formulas.  The formulas are shown in Figure~\ref{fig:abABacAC} 
and apply whenever $\min(o_1, o_2, o_3) \ge 2$.
Note that if we fix $o_1$ and let the other 
orders grow, $\scl$ is eventually constant.  This almost certainly reflects the 
fact that for $o_1$ fixed, eventually it is best to build an extremal 
surface using only rectangles, not group polygons, for the other factors.

\begin{figure}[ht]
\labellist
\small\hair 2pt
 \pinlabel {$1-\frac{1}{2o_2}-\frac{1}{2o_3}$} at 41 40
 \pinlabel {$1-\frac{\{1,3/4\}}{o_1}-\frac{1}{2o_2}$} at 130 40
 \pinlabel {$1-\frac{\{2,3/2\}}{2o_1}$} at 120 120
 \pinlabel {$1-\frac{\{1,3/4\}}{o_1}-\frac{1}{2o_3}$} at 41 120
 \pinlabel {$o_2$} at -10 71
 \pinlabel {$o_3$} at 74 -10
 \pinlabel {$\left(\left\{\frac{1}{2},\frac{2}{3}\right\}o_1,\left\{\frac{1}{2},\frac{2}{3}\right\}o_1\right)$} at 240 105
\endlabellist
\includegraphics[scale=0.9]{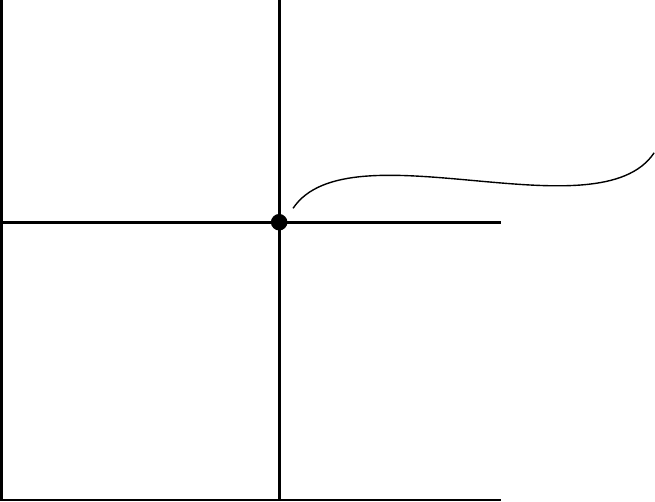}
\vspace{1mm}
\caption{For fixed $o_1$, the $o_2$-$o_3$ plane and piecewise formulas for 
$\scl(abABacAC)$.  All brackets refer to the residue class of $o_1$.  
Wherever the regions intersect in an integer point, the 
formulas agree.}
\label{fig:abABacAC}
\end{figure}
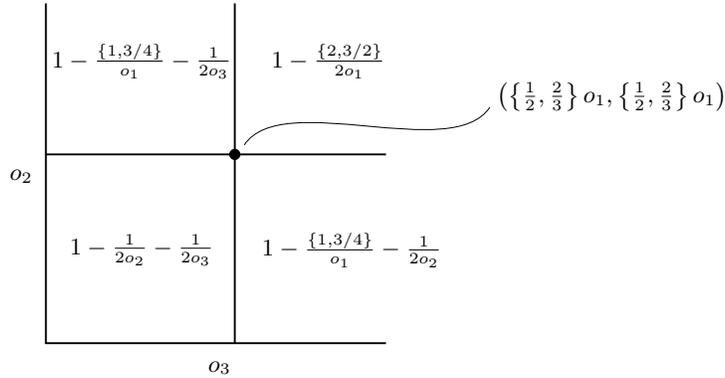

\subsection{Histograms}

Using \texttt{scylla}, it's simple to produce 
a histogram for free products of finite cyclic groups 
analogous to the histogram shown in Figure~\ref{fig:histogram}.
Figure~\ref{fig:scylla_histogram_23} shows a histogram of the 
$\scl$ spectrum in $\Z/3\Z * \Z/2\Z$ for many random 
words of length $40$, and Figure~\ref{fig:scylla_histogram_34} 
shows a histogram for many random words of length $30$ in $\Z/4\Z * \Z/3\Z$.  

\begin{figure}[ht]
\begin{center}
\labellist
\small\hair2pt
\pinlabel $\frac{1}{3}$ at 79 -16
\pinlabel $\frac{1}{2}$ at 137 -16
\pinlabel $\frac{2}{3}$ at 193 -16
\pinlabel $\frac{5}{6}$ at 252 -16
\pinlabel $1$ at 309 -15
\endlabellist
\includegraphics[scale=0.5]{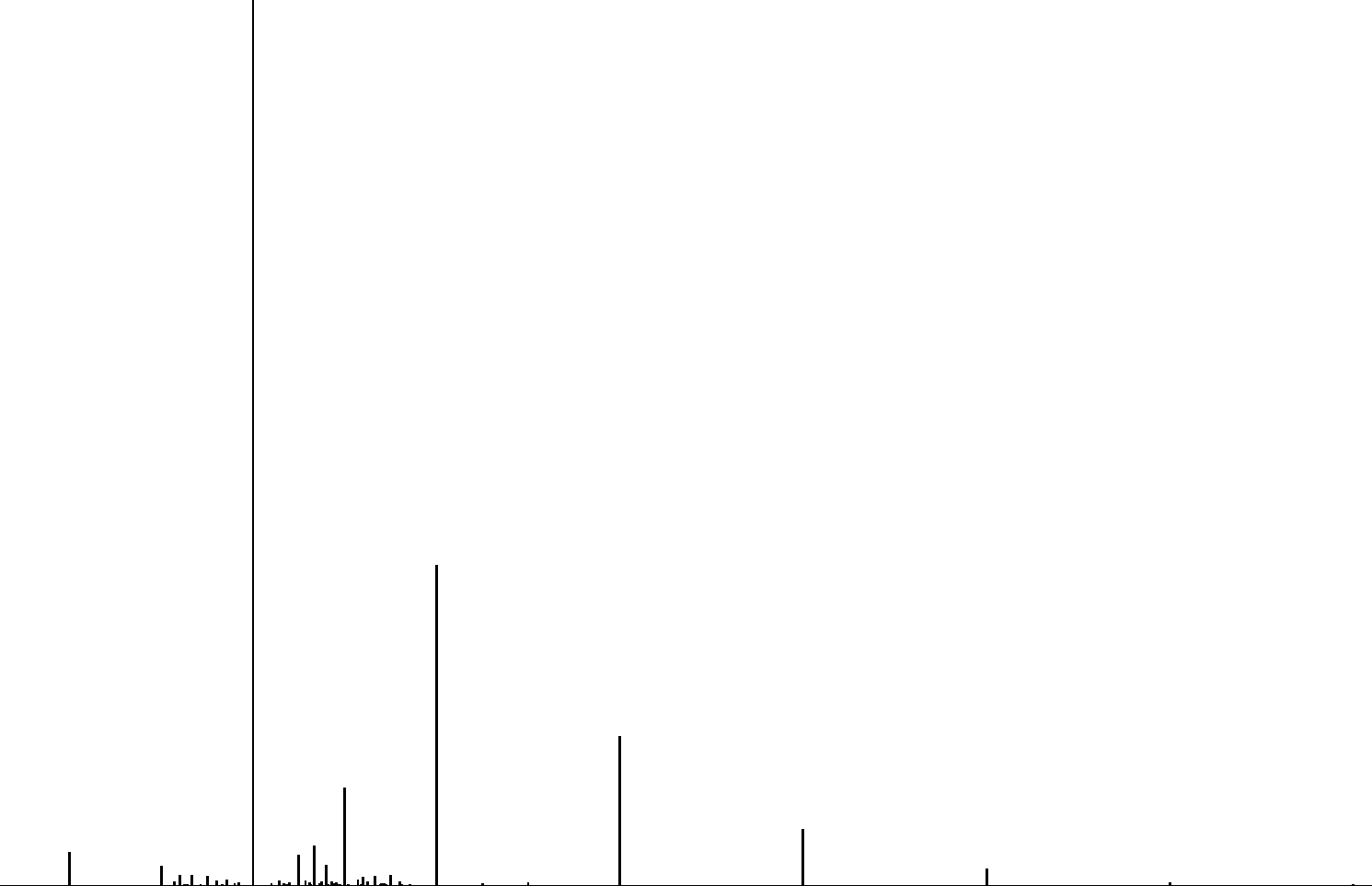}
\vspace{1mm}
\end{center}
\caption{A histogram of $\scl$ values of random words of length 
$40$ in $\Z/3\Z*\Z/2\Z$.}
\label{fig:scylla_histogram_23}
\end{figure}
\begin{figure}[ht]
\begin{center}
\labellist
\small\hair2pt
\pinlabel $\frac{5}{6}$ at 148 -20
\pinlabel $1$ at 208 -16
\pinlabel $\frac{7}{6}$ at 269 -20
\pinlabel $\frac{2}{3}$ at 86 -20
\endlabellist
\includegraphics[scale=0.45]{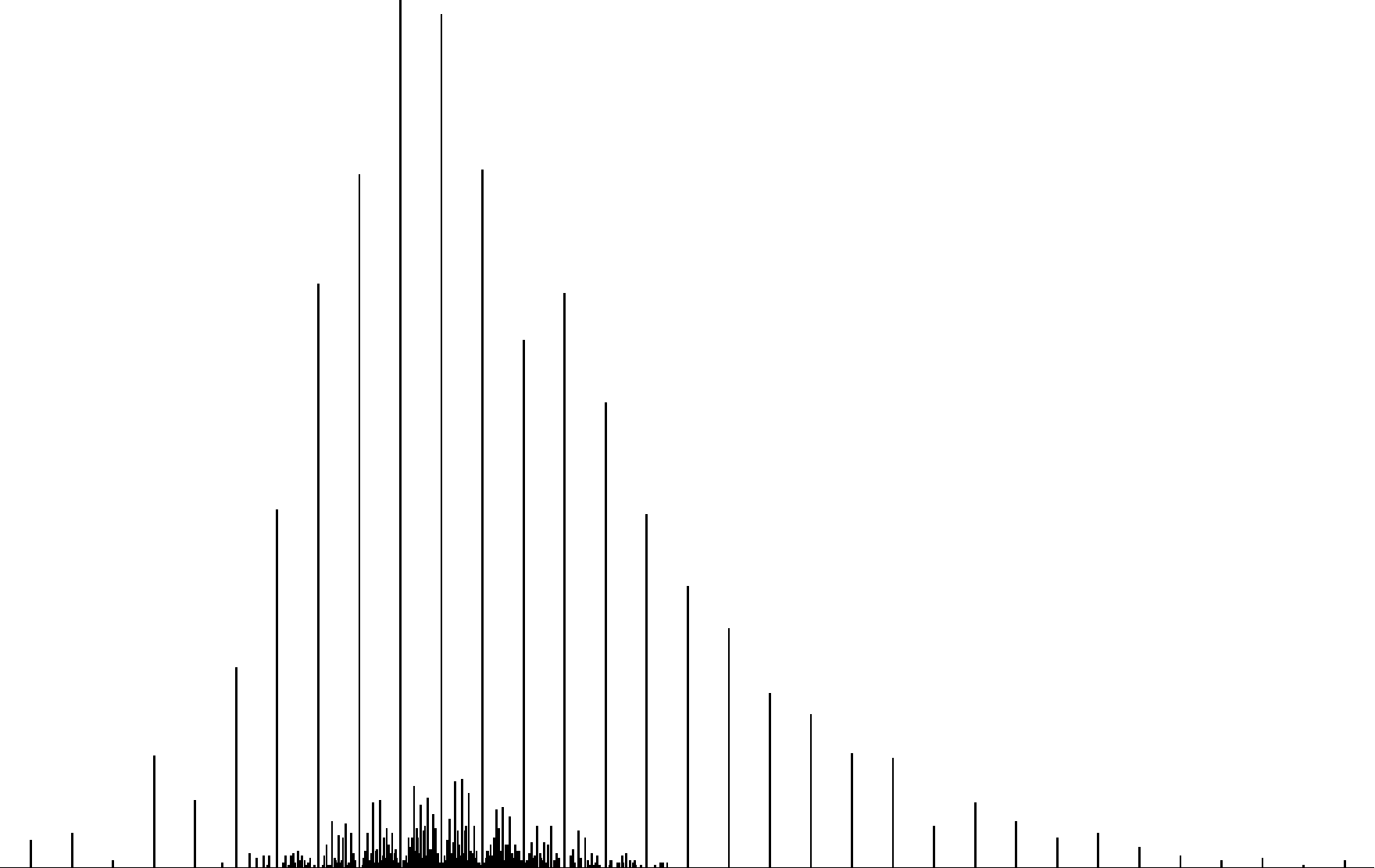}
\vspace{1mm}
\end{center}
\caption{A histogram of $\scl$ values of random words of length 
$30$ in $\Z/4\Z*\Z/3\Z$.  The gap between spikes is $1/24$.}
\label{fig:scylla_histogram_34}
\end{figure}

\subsection{Extensions}

As we mentioned in Remark~\ref{rem:proof_methods}, there are 
two ways to decompose surfaces: our method, and 
the technique which essentially builds surfaces out of 
van Kampen disks.  This latter method is messier in our situation, 
so we have avoided it.  However, it generalizes to more complicated 
group presentations, and from this perspective, it is 
clearer that the essential feature of free products of cyclic groups 
we have used is that there are no internal vertices in a reduced 
van Kampen diagram.

Consider a group presentation $G = \langle G \,| \, R\rangle$.  
We symmetrize $R$ and add in all cyclic conjugates.  A \emph{piece} 
in the presentation is word $u$ which occurs as a maximal common 
initial subword to two distinct words $r_1, r_2 \in R$.  Note that 
$r_1$ and $r_2$ may be the same cyclic word, but they are not allowed 
to be the same word.  A piece is simply an alignment of subwords 
in different relators.  Note that an internal edge in a van Kampen diagram 
is a part of a piece.

The key problem to compute $\scl$ as a linear programming problem is to understand 
how to compute the Euler characteristic of a surface knowing only the 
relators out of which it is built, and this is difficult in the presence 
of internal vertices in the van Kampen diagram.
If we require that the pieces in every relator are 
separated by at least one letter, for example, then 
the \texttt{scylla} algorithm strategy should still work, because 
internal vertices can be simplified away.  
However, this requirement is somewhat meaningless, because these group 
presentations are secretly just presentations of free products of 
cyclic groups!

\bibliography{biblio}{}
\bibliographystyle{plain}

\end{document}